\documentclass[11pt,reqno]{article}

\usepackage{booktabs} 
\usepackage{array} 
\usepackage{paralist} 
\usepackage{verbatim} 
\usepackage{amsthm}
\usepackage[table,xcdraw]{xcolor}
\usepackage[toc,title,page]{appendix}

\usepackage{latexsym, amsmath, amssymb, a4, epsfig, color}
\usepackage{blindtext}
\usepackage{graphicx}
\usepackage{subcaption}
\usepackage[utf8]{inputenc}
\usepackage[export]{adjustbox}
\usepackage{wrapfig}
\usepackage{apptools}

\usepackage{floatrow}

\usepackage{amssymb}
\usepackage{amsmath}
\usepackage[normalem]{ulem} 

\AtAppendix{\counterwithin{definition}{subsection}}
\AtAppendix{\counterwithin{theorem}{subsection}}

\graphicspath{}

\newtheorem{theorem}{Theorem}[section]
\newtheorem{cor}[theorem]{Corollary}
\newtheorem{lemma}[theorem]{Lemma}
\newtheorem{prop}[theorem]{Proposition}
\newtheorem{definition}{Definition}

\newtheorem{example}{Example}

\newcommand{\FF}{\mathbb{F}}
\newcommand{\RR}{\mathbb{R}}
\newcommand{\CC}{\mathbb{C}}
\newcommand{\NN}{\mathbb{N}}
\newcommand{\ZZ}{\mathbb{Z}}
\newcommand{\Om}{\Omega}
\newcommand{\ds}{\displaystyle}

\newcommand{\p}{\partial}
\newcommand{\pd}[2]{\frac {\p #1}{\p #2}}
\newcommand{\eqnref}[1]{(\ref {#1})}

\newcommand{\beq}{\begin{equation}}
\newcommand{\eeq}{\end{equation}}
\newcommand{\be}{\begin{equation*}}
\newcommand{\ee}{\end{equation*}}
\newcommand{\ba}{\begin{align*}}
\newcommand{\ea}{\end{align*}}
\newcommand{\bal}{\begin{align}}
\newcommand{\eal}{\end{align}}

\newcommand{\Kcal}{\mathcal{K}}

\numberwithin{equation}{section}
\numberwithin{figure}{section}

\begin{document}
\title{Construction of inclusions with vanishing generalized polarization tensors by imperfect interfaces\thanks{This study was supported by National Research Foundation of Korea (NRF) grant funded by the Korean government (MSIT) (NRF-2021R1A2C1011804).}
\author{Doosung Choi\thanks{Department of Mathematics, Louisiana State University, Baton Rouge, LA 70803, USA ({dchoi@lsu.edu}).} \footnotemark[3]
\and Mikyoung Lim\thanks{Department of Mathematical Sciences, Korea Advanced Institute of Science and Technology, Daejeon 34141, Republic of Korea ({mklim@kaist.ac.kr}).}}
}

\date{\today}
\maketitle
\begin{abstract}

We investigate the problem of planar conductivity inclusion with imperfect interface conditions. We assume that the inclusion is simply connected.
The presence of the inclusion causes a perturbation in the incident background field.
This perturbation admits a multipole expansion of which coefficients we call the generalized polarization tensors (GPTs), extending the previous terminology for inclusions with perfect interfaces. We derive explicit matrix expressions for the GPTs in terms of the incident field, material parameters, and geometry of the inclusion. As an application, we construct GPT-vanishing structures of general shape that result in negligible perturbations for all uniform incident fields. The structure consists of a simply connected core with an imperfect interface. We provide numerical examples of GPT-vanishing structures obtained by our proposed scheme.

\end{abstract}



\section{Introduction}


We consider the planar conductivity (or anti-plane elasticity) problem, where an inclusion is inserted into an infinite homogeneous background. The inclusion, having a different conductivity from that of the background,  generally induces perturbations to the background fields. 
However, certain types of inclusions do not cause any perturbation to the uniform background fields, known as neutral inclusions. Neutral inclusions have been extensively studied due to their possible applications in designing invisibility cloaking structures with metamaterials in various contexts, including acoustics, elasticity, electromagnetics, and microwaves \cite{Alu:2005:ATP, Ammari:2013:ENCm, Landy:2013:FPU, Liu:2017:NNS, Zhou:2006:DEW, Zhou:2007:AWT, Zhou:2008:EWT}.
Most results on neutral inclusions deal with perfect boundary conditions.
On the other hand, this paper addresses neutral inclusions with imperfect interfaces. More precisely, we provide a construction scheme for neutral inclusions of general shape with imperfect interfaces conditions where the core has finite low conductivity. 

Well-known examples of neutral inclusions (with perfect interfaces) that do not disturb incident uniform fields are coated disks and spheres with isotropic material parameters \cite{Hashin:1962:EMH, Hashin:1985:LIE, Hashin:1962:VAT, Jimenez:2013:NNI, Pham:2018:SCM, Pham:2019:MCF}.
Coated ellipses and ellipsoids exhibit neutrality to uniform fields of all directions when the conductivities of the core, shell, and matrix are appropriately selected (potentially anisotropic) \cite{Grabovsky:1995:MME1, Kerker:1975:IB, Milton:2002:TC, Sihvola:1999:EMF, Sihvola:1997:DPI}. It is worth noting that these are the only shapes that can achieve the neutral property for all uniform fields \cite{Kang:2014:CIF, Kang:2016:OBV, Milton:2001:NCI}. Non-elliptical coated inclusions can achieve neutrality for a single uniform field \cite{Jarczyk:2012:NCI, Milton:2001:NCI} (see also \cite{Lim:2020:IGS}).

One can design neutral inclusions using the concept of the generalized polarization tensors (GPTs), which are the coefficients of the multipole expansion for the perturbed potential caused by the inclusion; this concept extends the polarization tensor, introduced by Schiffer and Szeg\"{o} \cite{Schiffer:1949:VMP}, to higher orders.
GPT-vanishing structures--inclusions with vanishing values for the leading-order terms of the GPTs--are neutral inclusions in the asymptotic sense. Multi-coated concentric disks or balls have been reported to be GPT-vanishing structures \cite{Ammari:2013:ENC1, Wang:2013:MNC} (see also \cite{Ammari:2013:ENCh,Ammari:2013:ENCm} for results in the context of Helmholtz and Maxwell's equations). 
 GPT-vanishing structures of general shape were also obtained using a shape optimization approach  \cite{Feng:2017:CGV,Ji:2021:NIW, Kang:2022:EWN} (see also \cite{Choi:2023:GME}).

Imperfect interfaces are characterized by discontinuity in either the flux or the potential, in contrast to perfectly bonding boundaries of which both the flux and potential are continuous. The interface parameter, possibly a non-constant function, accounts for the discontinuity of the flux or the potential on the interface.
In \cite{Ru:1998:IDN}, Ru investigated neutral inclusions for the two-dimensional elastic problem with imperfect interfaces when the inclusion is stiffer than the background medium. 
For the conductivity problem, Benveniste and Miloh found neutral inclusions with imperfect interfaces for a single uniform field \cite{Benveniste:1999:NIC}.
In \cite{Kang:2019:CWN}, Kang and Li constructed PT-vanishing structures of general shape with imperfect interfaces,  assuming that the core is simply connected and perfectly conducting.

In this paper, we propose a new construction scheme of neutral inclusions (of general shape) for the planar conductivity problem with imperfect interfaces. The inclusions have discontinuity in the potential, where the core have the arbitrary constant conductivity. 
Different from \cite{Kang:2019:CWN}, the constant potential condition in the core is no longer assumed. Consequently, the construction of the interface parameter satisfying the neutrality condition becomes challenging. 
To address this problem, we find GPT-vanishing structures by employing the concept of the Faber polynomials polarization tensors (FPTs) (see \cite{Choi:2023:GME}). 
The FPTs are linear combinations of GPTs with coefficients given by Faber polynomials associated with the inclusion. We also refer to \cite{Faber:1903:PE, Faber:1907:PE} for details of Faber polynomials.
The concept of FPTs has been successfully applied to design semi-neutral inclusions \cite{Choi:2023:GME} (see also \cite{Choi:2023:IPP, Choi:2021:ASR,Jung:2021:SEL}). In this paper, we derive an explicit matrix expression of the FPTs for an inclusion with imperfect interface conditions and, by using the matrix expression, determine the interface parameter that vanishes several leading terms of the FPTs.

The rest of this paper is organized as follows. In Section \ref{sec:problem}, we present the problem formulation. In Section \ref{sec:Faber}, we introduce Faber polynomials based on conformal mappings and define the FPTs using integral formulation. Section \ref{section:main:FPTs} focuses on deriving the matrix expressions for the FPTs. The construction scheme for inclusions with vanishing GPTs (or FPTs) by imperfect interfaces is described in Section \ref{sec:numerical}. We also present numerical visualizations in Section \ref{sec:numerical}.

\section{Problem formulation}\label{sec:problem}
We let $D$ be a planar, simply connected domain with an analytic boundary. The core region $D$ has the constant conductivity $\sigma_c$ and the exterior region $\RR^2\setminus D$ is occupied by another homogeneous isotropic material with the constant conductivity $\sigma_m>0$. 
We consider the potential problem with imperfect interface condition:
\beq\label{LCproblem}
\begin{cases}
\ds \nabla \cdot{\sigma} \nabla u = 0\qquad&\mbox{in }\mathbb{\RR}^2, \\
\ds p\big(u\big|^+ - u\big|^-\big ) = \sigma_m \pd{u}{\nu}\Big|^+  \qquad&\mbox{on }\p D,\\[2mm]
\ds \sigma_m \pd{u}{\nu}\Big|^+ = \sigma_c \pd{u}{\nu}\Big|^-  \qquad&\mbox{on }\p D,\\[2mm]
\ds (u-H)(x)  =O({|x|^{-1}})\qquad&\mbox{as } |x| \to \infty,
\end{cases}
\eeq
where $p$ is the interface parameter that is a real-valued function on $\p\Om$ and $H$ is an arbitrary entire harmonic function.  The conductivity distribution ${\sigma}$ has the form
$$
{\sigma} = \sigma_c\chi_{D} + \sigma_m\chi_{\RR^2\setminus\overline{D}},
$$
where $\chi$ is the characteristic function. 

\subsection{Generalized polarization tensors}
For a Lipschitz domain $D$, we define the single and double layer potentials with a density $\varphi\in L^2(\p D)$,
\begin{align*}
\mathcal{S}_{\p D}[\varphi](x)&=\int_{\partial D}\Gamma(x-y)\varphi(y)\,d\sigma(y),\quad x\in\RR^2,\\[1mm]
\mathcal{D}_{\p D}[\varphi](x)&=\int_{\partial D}\frac{\partial}{\partial\nu_y}\Gamma(x-y)\varphi(y)\,d\sigma(y),\quad x\in\RR^2\setminus\partial D,
\end{align*}
where  $\Gamma$ is the fundamental solution to the Laplacian, {\it i.e.}, $\Gamma(x)=\frac{1}{2\pi}\ln|x|$, and $\nu_y$ denotes the outward unit normal vector on $\partial D$. We also define the so-called Neumann--Poincar\'{e} (NP) operator
$$
\mathcal{K}_{\p D}[\varphi](x)=p.v.\,\int_{\partial D}\frac{\partial}{\partial\nu_y}\Gamma(x-y)\varphi(y)\,d\sigma(y),\quad x\in \partial D,
$$
and its $L^2$-adjoint operator
\begin{align*}
\mathcal{K}^*_{\p D}[\varphi](x)=p.v.\,\int_{\partial D}\frac{\partial}{\partial\nu_x}\Gamma(x-y)\varphi(y)\,d\sigma(y),\quad x\in \partial D.
\end{align*}
Here, $p.v. $ stands for the Cauchy principal value. The operator $\mathcal{K}^*_{\p D}$ is bounded on $L^2(\p D)$. 
We identify $x=(x_1,x_2)\in\RR^2$ with $z=x_1+\mathrm{i}x_2\in\CC$. 
We set 
$$
\mathcal{S}_{\p D}[\varphi](z):=\mathcal{S}_{\p D}[\varphi](x).
$$
Likewise, we define $\mathcal{D}_{\p D}[\varphi](z)$ and $\mathcal{K}^*_{\p D}[\varphi](z)$.

On the interface $\p D$, the layer potentials satisfy the jump relations
\begin{align*}
\mathcal{S}_{\p D}[\varphi]\Big|^{+}&=\mathcal{S}_{\p D}[\varphi]\Big|^{-},\\[1mm]
\frac{\partial}{\partial\nu}\mathcal{S}_{\p D}[\varphi]\Big|^{\pm}&=\left(\pm\frac{1}{2}I+\Kcal^*_{\p D}\right)[\varphi],\\[1mm]
\mathcal{D}_{\p D}[\varphi]\Big|^{\pm}&=\left(\mp\frac{1}{2}I+\Kcal_{\p D}\right)[\varphi],\\[1mm]
\frac{\partial}{\partial\nu}\mathcal{D}_{\p D}[\varphi]\Big|^+&=\frac{\partial}{\partial\nu}\mathcal{D}_{\p D}[\varphi]\Big|^-.
\end{align*}
Here, $I$ denotes the identity operator on $L^2(\p D)$.

Let $\Re[\cdot]$ and $\Im[\cdot]$ denote the real and imaginary part of a complex number, respectively. The solution $u$ of the system \eqnref{LCproblem} can be represented  by
\beq\label{u_representation}
u(z) = H(z) + \left( \mathcal{S}_{\p D} \, p- \tau\mathcal{D}_{\p D} \right)[\varphi](z), \quad (\Re(z),\Im(z))\in\RR^2\setminus{\p D}
\eeq
with $\lambda=\frac{\sigma_c+\sigma_m}{2(\sigma_c-\sigma_m)}$ and $\tau = \frac{\sigma_c \sigma_m}{\sigma_c-\sigma_m},$
where $p$ denotes the multiplication operator by $p$ (by abuse of the notation) and the density function $\varphi$ is the solution to the integral equation
\beq\label{varphi}
\left(  \left(\lambda I-\Kcal^*_{\p D}\right) p + \tau\frac{\p}{\p\nu} \mathcal{D}_{\p D} \right)\left[ \varphi \right] =\frac{\p H}{\p \nu}.
\eeq
This representation was introduced in \cite{Kang:2019:CWN} for the case $\sigma_c=\infty$ (that is, $\tau=1$) with arbitrary smooth function $p>0$.
Following the proof in \cite{Kang:2019:CWN}, one can show that the operator $\left(\lambda I-\Kcal^*_{\p D}\right) p + \tau\frac{\p}{\p\nu} \mathcal{D}_{\p D}$ is invertible for $0<\sigma\neq 1<\infty$.
By the jump relations,  the solution \eqnref{u_representation} meets all conditions of the system \eqnref{LCproblem}.

\begin{definition}\label{def:GPTs}
 Set $Z_n(\eta)=\eta^n$ for each $n\in\NN$. For $m,n\in\NN$, we define
\begin{align*}
\NN_{mn}^{(1)}(D,p,\sigma_c,\sigma_m)&=\int_{\p D} \bigg(Z_n(\eta)\, p(\eta) - \tau \frac{\p Z_n(\eta)}{\p\nu} \bigg) \bigg( (\lambda I-\Kcal^*_{\p D}) p + \tau \frac{\p\mathcal{D}_{\p D}}{\p\nu} \bigg)^{-1}\left[\frac{\p Z_m }{\p\nu} \right](\eta) \,d\sigma(\eta),\\
\NN_{mn}^{(2)}(D,p,\sigma_c,\sigma_m)&=\int_{\p D} \bigg( Z_n(\eta)\, p(\eta) - \tau \frac{\p Z_n(\eta)}{\p\nu} \bigg) \bigg( (\lambda I-\Kcal^*_{\p D}) p + \tau \frac{\p\mathcal{D}_{\p D}}{\p\nu} \bigg)^{-1}\left[\frac{ \p \overline{Z_m}}{\p \nu} \right](\eta) \,d\sigma(\eta),
\end{align*}
where $\lambda=\frac{\sigma_c+\sigma_m}{2(\sigma_c-\sigma_m)}$ and $\tau = \frac{\sigma_c \sigma_m}{\sigma_c-\sigma_m}$.
We call $\NN_{mn}^{(1)}$ and $\NN_{mn}^{(2)}$ the (complex) generalized polarization tensors (GPTs) corresponding to the inclusion $D$ with the interface function $p$ and the conductivities $\sigma_c$, $\sigma_m$.
\end{definition}

The GPTs generalize the polarization tensor (PT) that was introduced by Schiffer and Szeg\"{o} \cite{Schiffer:1949:VMP}. The concept of the PT, the GPTs for $m=n=1$, is applicable to the potential problem with imperfect interface condition.

\subsection{Multipole expansion}

Using this Taylor expansion, the solution expression \eqnref{u_representation} and the definition of the GPTs, one can show the following multipole expansion for the inclusion problem with imperfect interface (refer to \cite{Ammari:2013:MSM,Ammari:2003:ESS,Ammari:2001:AFP,Vogelius:2000:AFP} for the perfect interface problem).
\begin{prop}\label{prop:multipole}
Let $u$ be the solution to \eqnref{LCproblem} for $H(z) = \Re \left[ \sum_{m=1}^\infty\alpha_m z^m \right]$ with complex coefficients $\alpha_m$. Let $\NN_{mn}^{(1)}$ and $\NN_{mn}^{(2)}$ be given by Definition \ref{def:GPTs}. We have
\begin{align}\notag
u(z)=H(z)- \Re\left[\sum_{n=1}^\infty \sum_{m=1}^\infty\frac{1}{4\pi n} \left( \alpha_m \NN_{mn}^{(1)}+\overline{\alpha_m}\, \NN_{mn}^{(2)}
\right){z^{-n}} \right],\quad |z|>\sup\left\{|y|:y\in D\right\}.
\end{align}
\end{prop}
\begin{proof}
By the Taylor expansion of the logarithmic function, we get as $|z|\to\infty$,
$$
\Gamma(z-\eta) = \frac{1}{2\pi} \Re\left[\ln(z-\eta) \right] = \frac{1}{2\pi} \Re\bigg[\ln z - \sum_{n=1}^\infty \frac{1}{n} \frac{\eta^n}{z^n} \bigg] = \Gamma(z) - \sum_{n=1}^\infty \frac{1}{2\pi n} \Re\left[Z_n(\eta) \, z^{-n} \right]
.
$$
Using this expansion and \eqnref{u_representation}, the solution $u$ satisfies
\begin{align}
u(z)
&= H(z) +  \mathcal{S}_{\p D} [p \varphi](z) - \tau\mathcal{D}_{\p D}[\varphi](z) \notag \\
&= H(z) + \int_{\p D} \left[\Gamma(z-\eta)\, p(\eta) - \tau\frac{\p \Gamma(z-\eta)}{\p \nu_\eta} \right]\varphi(\eta) \, d\sigma(\eta) \notag \\
&= H(z) + \int_{\p D}  \Gamma(z)  \, p(\eta) \, \varphi(\eta) \, d\sigma(\eta) \notag \\
& \quad - \sum_{n=1}^\infty \frac{z^{-n}}{2\pi n} \int_{\p D} \Re\left[Z_n(\eta) \, p(\eta) - \tau \frac{\p Z_n(\eta)}{\p \nu_\eta} \right]\varphi(\eta) \, d\sigma(\eta) \notag \\
&= H(z) - \sum_{n=1}^\infty \frac{z^{-n}}{2\pi n} \int_{\p D} \Re\left[Z_n(\eta) \, p(\eta) - \tau \frac{\p Z_n(\eta)}{\p \nu_\eta} \right]\varphi(\eta) \, d\sigma(\eta).\label{multipole}
\end{align}
The last equation holds because $\int_{\p D} p(\eta) \, \varphi(\eta) \, d\sigma(\eta) = 0$.  Due to \eqnref{varphi}, we have 
$$
\varphi = \left( \left(\lambda I-\Kcal^*_{\p D}\right) p + \tau\frac{\p}{\p\nu} \mathcal{D}_{\p D} \right)^{-1}\left[\Re \left( \sum_{m=1}^\infty\alpha_m \frac{\p Z_m}{\p \nu}  \right) \right],
$$
and substituting this $\varphi$ into  \eqnref{multipole} leads the desired result.
\end{proof}

%

In \cite{Kang:2019:CWN}, the interface function $p$ was constructed such that the corresponding PT, i.e., $\NN_{11}^{(1)}$ and $\NN_{11}^{(2)}$, vanishes under the assumption that $\sigma_c=\infty$ (that is, the core is occupied with perfectly conducting material). 

In the present paper, we construct the interface function $p$ such that the GPTs vanish up to some finite orders (including the PT), where the core $D$ is a given domain with arbitrary constant conductivity $\sigma_c$. 
In view of Proposition \ref{prop:multipole}, the resulting inclusion is a neutral inclusion in the asymptotic sense. 
For the construction of $p$, we will apply the complex potential approach in  \cite{Choi:2023:GME}, where the concept of the Faber polynomial polarization tensors was employed to build semi-neutral inclusions of general shape. 


\section{Faber polynomial polarization tensors (FPTs)}\label{sec:Faber}
From the Riemann mapping theorem, there uniquely exist $\gamma > 0$ and conformal mapping $\Psi$ from $\{w\in\CC:|w|>\gamma\}$ onto $\CC\setminus \overline{D}$ such that
\beq\label{conformal:Psi}
\Psi(w)=w+a_0+\frac{a_1}{w}+\frac{a_2}{w^2}+\cdots.
\eeq
The quantity $\gamma$ is called the conformal radius of $D$.
One can numerically compute the numerical computation of $a_n$ from a given parametrization of $\p\Om$ (see, for instance, \cite{Jung:2021:SEL}). As an univalent function, $\Psi$ defines the Faber polynomials $\{F_m\}_{m=1}^\infty$ by the relation \cite{Faber:1903:PE} 
\beq\label{eqn:Fabergenerating}
\frac{\Psi'(w)}{\Psi(w)-z}=\sum_{m=0}^\infty \frac{F_m(z)}{w^{m+1}},\quad z\in{\overline{D}},\ |w|>\gamma.
\eeq
The Faber polynomials $F_m$ are monic polynomials of degree $m$ that are uniquely determined by the conformal mapping coefficients $\{ a_n \}_{0 \le n \le m-1}$ via the recursive relation \cite{Duren:1983:UF}:
\beq\label{Faberrecursion}
F_{m+1} (z) = z F_m (z) - m a_m - \sum_{n=0} ^{m} a_n F_{m-n} (z), \quad m\ge 0.
\eeq 
In particular, we have
\be
\begin{aligned}
F_0(z)&=1,\quad
F_1(z)=z-a_0,\quad
F_2(z)=z^2-2a_0 z+a_0^2-2a_1,\\
F_3(z)& = z^3 - 3a_0z^2 + 3(a_0^2-a_1)z - a_0^3+3a_0a_1-3a_2.
\end{aligned}
\ee
For $z=\Psi(w)$ with $|w|>r$, it holds that
\begin{equation} \label{eqn:Faberdefinition}
F_m(\Psi(w))=w^m+\sum_{n=1}^{\infty}c_{mn}{w^{-n}}
\end{equation}
with the so-called Grunsky coefficient $c_{mn}$.
Recursive relations for the Faber polynomial coefficients and the Grunsky coefficients are well-known (see, for instance, \cite{Duren:1983:UF,Grunsky:1939:KSA}).

 We introduce the modified polar coordinates $(\rho,\theta)\in[\rho_0,\infty)\times [0,2\pi)$, $\rho_0=\ln \gamma$,  for $z\in \CC\setminus D$ via the relation 
$$z=\Psi(w)=\Psi(e^{\rho+i\theta}).$$
We denote the scale factor as $h=\left|\frac{\partial \Psi}{\partial\rho}\right|=\left|\frac{\partial \Psi}{\partial\theta}\right| = e^\rho \left|\Psi'\right|.$ One can easily see that $d\sigma(z)=h(\rho_0,\theta)d\theta$ on $\p D$. For $z=\Psi(w)\in \p D$, we have
\beq\label{NDchange}
\pd{u}{\nu}\Big|_{\p D}^{\pm} = \frac{1}{h(\rho_0,\theta)}\frac{\partial }{\partial \rho}u(\Psi(w))\Big|_{\rho\rightarrow\rho_0^\pm}.
\eeq

We assume that $\Om$ has an analytic boundary, that is, $\Psi$ can be conformally extended to $\{w\in\CC:|w|>\gamma-\delta\}$ for some $\delta>0$. Then \eqnref{eqn:Fabergenerating} holds with a modified conformal radius and domain. The Grunsky coefficients satisfy the so-called {\it strong and week Grunsky inequalities}. In particular, it holds that
$
\left|c_{mn}\right|\leq 2 m \gamma^{m+n}.
$
The same relation holds with $\gamma-\delta$ instead of $\gamma$, that is,
\beq \label{C:decay}
\left|c_{mn}\right|\leq 2 m (\gamma-\delta)^{m+n}.
\eeq
The Faber polynomials form a basis for complex analytic functions in $D$ \cite{Duren:1983:UF}.
Let $v$ be a complex analytic function in $D$ and continuous in $\overline{D}$. By applying the Cauchy integral formula to $v$ and applying \eqnref{eqn:Fabergenerating}, it follows that
\beq\notag
v(z)=\sum_{m=0}^{\infty}b_m F_m(z)\quad\mbox{in } D
\eeq
with
$
b_m = \frac{1}{2\pi i}\int_{|w|=\gamma}\frac{v(\Psi(w))}{w^{m+1}}\,dw.
$
Hence, we have 
\beq\label{B:decay}
|b_m|\leq \|v\|_{L^\infty(\p D)}\gamma^{-m}.
\eeq

\subsection{FPTs and FPTs-vanishing structure}\label{subsec:FPT_vanishing}

Since the Faber polynomials form a basis for complex analytic functions, $H$ is given by
\beq\label{H}
H(z) = \sum_{m=1}^\infty \Re \big[\alpha_m F_m(z) \big]
\eeq
with some complex coefficients $\alpha_m$.  

Let $u$ be the solution to the imperfect interface problem \eqnref{LCproblem}.  We can expand $u\big|_D$ into the real parts of the Faber polynomials. Furthermore, as the solution to  \eqnref{LCproblem} linearly depends on the background solution $H$, we have
\beq\label{uext:int}
\begin{aligned}
u(z)&=\sum_{m=1}^\infty\Re\bigg[  \sum_{n=1}^\infty \beta_{mn} F_n(z) \bigg]\quad \mbox{in } D
\end{aligned}
\eeq
with some constants $\beta_{mn}$.
As $u-H$ is decaying near infinity and $\Psi$ is conformal, it holds that for $z=\Psi(w)$,
\begin{align}\label{uext:ext}
u(z)=
H(z) + \sum_{m=1}^\infty \Re\bigg[\sum_{n=1}^\infty s_{mn} w^{-n} \bigg] \quad \mbox{in } \RR^2 \setminus \overline{D}
\end{align}
with some constants $s_{mn}$.  Following \cite{Choi:2023:GME}, we call the expansion in \eqnref{uext:ext} the {\it geometric multipole expansion} of $u$.
For each $m$, both $\beta_{mn}$ and $s_{mn}$ are determined by $\alpha_m$. In particular, $\beta_{mn}=s_{mn}=0$ if we give $\alpha_m=0$ for some $m$ in \eqnref{H}.  In fact, we can decompose the coefficient $s_{mn}$ by using  some quantities defined as below.

\begin{definition}
Let $F_n$ be the Faber polynomial for each $n\in \NN$.  For $m,n\in\NN$, we define
\begin{align*}
\FF_{mn}^{(1)}(D,p,\sigma_c,\sigma_m)&=\int_{\p D} \bigg(F_n(\eta)\, p(\eta) - \tau \frac{\p F_n(\eta)}{\p\nu} \bigg) \bigg( (\lambda I-\Kcal^*_{\p D}) p + \tau \frac{\p\mathcal{D}_{\p D}}{\p\nu} \bigg)^{-1}\left[\frac{\p F_m }{\p\nu} \right](\eta) \,d\sigma(\eta),\\
\FF_{mn}^{(2)}(D,p,\sigma_c,\sigma_m)&=\int_{\p D} \bigg( F_n(\eta)\, p(\eta) - \tau \frac{\p F_n(\eta)}{\p\nu} \bigg) \bigg( (\lambda I-\Kcal^*_{\p D}) p + \tau \frac{\p\mathcal{D}_{\p D}}{\p\nu} \bigg)^{-1}\left[\frac{ \p \overline{F_m}}{\p \nu} \right](\eta) \,d\sigma(\eta),
\end{align*}
where $\lambda=\frac{\sigma_c+\sigma_m}{2(\sigma_c-\sigma_m)}$ and $\tau = \frac{\sigma_c \sigma_m}{\sigma_c-\sigma_m}$.  We call  $\FF_{mn}^{(1)}$ and $\FF_{mn}^{(2)}$ the {\it Faber polynomial polarization tensors (FPTs)} for an inclusion with an imperfect interface, by extending the definition for an inclusion with perfect interface conditions in \cite{Choi:2023:GME}.
\end{definition}

For $z=\Psi(w)$, integrating the expansion \eqnref{eqn:Fabergenerating} with respect to $w$ yields that
$$
\frac{1}{2\pi} \ln(z-\eta) = -\sum_{m=0}^\infty \frac{F_m(\eta)}{2\pi m}\frac{1}{w^{m}} + \text{constant},\quad \eta \in{\overline{D}},\ |w|>\gamma.
$$
By using the linear dependence of $u$ on $H$ and the above series expansion, we can decompose $s_{mn}$ as
\beq\label{s}
s_{mn} = -\alpha_m \frac{\FF_{mn}^{(1)}}{4\pi n} - \overline{\alpha_m}\, \frac{\FF_{mn}^{(2)}}{4\pi n},
\eeq
where $\FF_{mn}^{(1)}$ and $\FF_{mn}^{(2)}$ are the FPTs independent of $\alpha_m$, namely,
\begin{align}\label{GME}
u(z) = H(z) - \sum_{m=1}^\infty \Re\bigg[\sum_{n=1}^\infty \frac{1}{4\pi n} \left( \alpha_m \FF_{mn}^{(1)} + \overline{\alpha_m} \FF_{mn}^{(2)} \right) w^{-n} \bigg] \quad \mbox{in } \RR^2 \setminus \overline{D},
\end{align}
where $H$ is given by an entire Faber series $H(z) = \sum_{m=1}^\infty \Re \big[\alpha_m F_m(z) \big]$.  We omit the derivation of \eqnref{GME} because it is very similar with the proof of Proposition \ref{prop:multipole}.

The first-order terms of the FPTs correspond to the PTs considered in \cite{Kang:2019:CWN}. 
More precisely, the polarization tensor $M=\{m_{ij}\}_{i,j=1,2}$ admits the expression
$$
M = 
\left(
\begin{array}{cc}
m_{11} & m_{12}\\
m_{21} & m_{22}
\end{array}
\right)
=
\frac{1}{2}
\left(
\begin{array}{cc}
\ds\Re
\left[\FF_{11}^{(1)} + \FF_{11}^{(2)}\right]& \ds\Im\left[\FF_{11}^{(1)} + \FF_{11}^{(2)}\right]\\[2mm]
\ds\Im\left[\FF_{11}^{(1)} - \FF_{11}^{(2)}\right] & \ds\Re\left[\FF_{11}^{(2)} - \FF_{11}^{(1)}\right]
\end{array}
\right).
$$

We can express the Faber polynomial $F_m(z)$ as
\beq\label{Fm:pmn}
F_m(z) = \sum_{n=0}^{m} q_{mn} z^n,
\eeq
where for a fixed $m$, the coefficient $q_{mn}$ depends only on $\{ a_k \}_{0 \le k \le m-1}$. One can easily obtain recursive formulas for $q_{mn}$ from \eqnref{Faberrecursion}. 
From \eqnref{Fm:pmn} and the definition of the FPTs, it holds for each $m,n$ that 
\beq\label{FPN}
\begin{aligned}
\FF_{mn}^{(1)} &= \sum_{k=1} ^{m} \sum_{l=1} ^{n} q_{mk} \, q_{nl}\, \mathbb{N}_{kl}^{(1)}, \\
\FF_{mn}^{(2)} &= \sum_{k=1} ^{m} \sum_{l=1} ^{n} \overline{q_{mk}}\,  q_{nl}\, \mathbb{N}_{kl}^{(2)}.
\end{aligned}
\eeq

From the relations of the GPTs and FPTs, we can then obtain matrix factorizations for the GPTs.
Set $Q=(q_{mn})_{m,n=1}^\infty$ with $q_{mn}$ given by \eqnref{Fm:pmn}. 
We can rewrite relations \eqnref{FPN} in matrix form as 
\beq\label{FPNP}
\begin{aligned}
\FF^{(1)} &= Q\, \mathbb{N}^{(1)} Q^T,\\
\FF^{(2)} &= \overline{Q}\, \mathbb{N}^{(2)} Q^T,
\end{aligned}
\eeq
where $\overline{Q}$ and $Q^T$ denote the conjugate and transpose matrices of $Q$, respectively.
Using the recursive relation in \eqnref{Faberrecursion}, we get that, for each $m\geq1$,
\begin{align}\label{inititalP}
&q_{mm} = 1,\quad q_{(m+1)m} = - (m+1) a_0, \quad
q_{mn}=0\quad\mbox{for all }n\geq m+1.
\end{align}
Indeed,
\beq\label{def:P}
Q=
\left[
\begin{matrix}
1 & 0 & 0   \quad  \cdots\\[1.5mm]
-2a_0 & 1 & 0   \quad \cdots\\[1.5mm]
3a_0^2-3a_1 & -3a_ 0 &  1  \quad \cdots\\[1.5mm]
\vdots & \vdots &  \vdots \quad \ddots 
\end{matrix}
\right],
\eeq
and then $Q$ is lower triangular and invertible.  Consequently,  \eqnref{FPNP} gives that
\beq\label{NQFQ}
\begin{aligned}
\mathbb{N}^{(1)} &= Q^{-1} \, \FF^{(1)} \left(Q^{-1}\right)^T,\\
\mathbb{N}^{(2)} &= \overline{Q^{-1}} \, \FF^{(2)} \left( Q^{-1} \right)^T.
\end{aligned}
\eeq
Indeed, the first terms of CGPTs and FPTs are always identical:
$$
\NN_{11}^{(1)} = \FF_{11}^{(1)}, \quad \NN_{11}^{(2)} = \FF_{11}^{(2)}.
$$
Furthermore, if the center of the conformal mapping in \eqnref{conformal:Psi} is given by $0$, namely, $a_0=0$,  the lower diagonal elements $Q$ in \eqnref{def:P} are all zero.  In this case,  CGPTs and FPTs have the same leading terms of order up to 2:
$$
\NN_{mn}^{(1)} = \FF_{mn}^{(1)}, \quad \NN_{mn}^{(2)} = \FF_{mn}^{(2)},  \quad 1 \le m,n\le2.
$$

Recall that our aim is to construct $p$ such that the corresponding GPTs vanish up to some finite orders. In view of \eqnref{FPN} the GPT-vanishing structures can be equivalently achieved by the FPT-vanishing structure.
It is worth remarking that for the instance of the perfectly bonding interface, explicit formulas of the FPTs were derived in \cite{Choi:2023:GME}.

\section{Matrix expressions for the FPTs}\label{section:main:FPTs}
Let $H$ be an entire harmonic function given by \eqnref{H} and $u$ be the solution to \eqnref{LCproblem}. 
From \eqnref{uext:int}, \eqnref{uext:int}, \eqnref{H} and the fact that $\Psi$ admits an analytic extension for $\rho\in [\rho_0-\delta,\rho_0]$, we have for $z=\Psi(w)$ that
\begin{align}
u(z)
&=
\begin{cases}
\ds\Re\bigg[  \sum_{m=1}^\infty\sum_{n=1}^\infty \beta_{mn} F_n(z) \bigg]\quad &\mbox{for } \rho\in[\rho_0-\delta,\rho_0],\\
\ds  \Re\bigg[ \sum_{m=1}^\infty\alpha_m F_m(z) +\sum_{m=1}^\infty\sum_{n=1}^\infty s_{mn} w^{-n} \bigg] \quad &\mbox{for }\rho>\rho_0
\end{cases} \notag \\ 
&= 
\begin{cases}
\ds \Re\bigg[ \sum_{m=1}^\infty\sum_{n=1}^\infty \beta_{mn} w^n +\sum_{m=1}^\infty \sum_{n=1}^\infty\sum_{l=1}^\infty \beta_{ml} c_{ln} w^{-n}  \bigg]\quad&\mbox{for } \rho\in[\rho_0-\delta,\rho_0],\\
\ds  \Re\bigg[\sum_{m=1}^\infty\alpha_m w^m +\sum_{m=1}^\infty\sum_{n=1}^\infty \left(\alpha_m c_{mn} +s_{mn}\right)w^{-n} \bigg] \quad &\mbox{for }\rho>\rho_0 \label{uint}
\end{cases}
\end{align}
Since we only consider the background solution $H$ of finite order, we can assume finite summation for $m$. 
We may assume a finite summation for $m$ since only finite order polynomials are considered for $H$. For decaying properties of $\beta_{mn}$ and $c_{ln}$, we refer to \eqnref{C:decay} and \eqnref{B:decay}.

The following lemma is useful to derive equivalent coefficients relations for a complex series to satisfy a boundary condition.
\begin{lemma}\label{Recoeff}
Let $\sum_{n=-\infty}^\infty c_n w^n$ be a complex analytic function in a neighborhood of $\{|w|=\gamma\}$. Then
$
\Re\left[ \sum_{n=-\infty}^\infty c_n w^n \right] =0$ on $|w|=\gamma$ if and only if $c_{-n} + \overline{c_n} \gamma^{2n}=0$ for each $n\geq 0$.
\end{lemma}
\begin{proof}
 Since $\overline{w} = \gamma^2 w^{-1}$ on $|w|=\gamma$, we have
$
2\Re\left[ \sum_{n=-\infty}^\infty c_n w^n \right] 
= \sum_{n=-\infty}^\infty \left(c_n w^n + \overline{c_n w^n}\right) 
= \sum_{n=-\infty}^\infty \left(c_{-n} + \overline{c_n} \gamma^{2n}\right) w^{-n}.
$
This proves the lemma.
\end{proof}

\subsection{Equivalent relations for the boundary conditions}
Let us define semi-infinite matrices
\beq\label{Gnotation}
\begin{aligned}
\boldsymbol{\alpha} &= \big\{\alpha_m \delta_{mn}\big\}_{m,n\ge1}, \quad\boldsymbol{\beta}=\big\{\beta_{mn}\big\}_{m,n\geq1},\quad \boldsymbol{s} = \big\{s_{mn} \big\}_{m,n\ge1},\\
\mathcal{N}&=\big[n\delta_{mn}\big]_{m,n\geq 1}, \quad\gamma^{\tau\mathcal{N}} = \big\{\gamma^{\tau n} \delta_{mn}\big\}_{m,n\ge1},  \quad C = \big\{c_{mn} \big\}_{m,n\ge1},
\end{aligned}
\eeq
where $\delta_{mn}$ is the Kronecker-Delta function and $\tau\in\RR$.

\begin{lemma} For $u$ given by \eqnref{uint}, the condition $\pd{u}{\nu}\left|^+=\sigma_c \pd{u}{\nu}\right|^-$ is equivalent to the relation between the coefficients:
\beq\label{abcF1}
\left(\boldsymbol{\alpha} C +\boldsymbol{s}  \right) \gamma^{-2\mathcal{N}} =  \overline{\boldsymbol{\alpha}} - \sigma \overline{\boldsymbol{\beta}} + \sigma \boldsymbol{\beta} C \gamma^{-2\mathcal{N}}.
\eeq
\end{lemma}
\begin{proof}
In view of \eqnref{NDchange}, the Neumann boundary condition
$\sigma_m\pd{u}{\nu}\big|^+=\sigma_c \pd{u}{\nu}\big|^-$
is equivalent to 
\beq\label{BC:equi1}
\sigma_m \frac{\partial }{\partial \rho}u(\Psi(w))\Big|_{\rho\rightarrow\rho_0^+} = \sigma_ c \frac{\partial }{\partial \rho}u(\Psi(w))\Big|_{\rho\rightarrow\rho_0^-}.
\eeq
From \eqnref{uint}, we obtain
\begin{align*}
\frac{\partial }{\partial \rho}u(\Psi(w))\Big|_{\rho\rightarrow\rho_0^+}
&=\Re\bigg[ \sum_{m=1}^\infty m \alpha_m w^{m} - \sum_{m=1}^\infty \sum_{n=1}^\infty n \left(\alpha_m c_{mn} +s_{mn}\right)w^{-n}  \bigg]\bigg|_{\rho=\rho_0},\\
\frac{\partial }{\partial \rho}u(\Psi(w))\Big|_{\rho\rightarrow\rho_0^-}
&= \Re\bigg[ \sum_{m=1}^\infty \sum_{n=1}^\infty n\beta_{mn} w^{n} 
- \sum_{m=1}^\infty \sum_{n=1}^\infty \sum_{l=1}^\infty n\beta_{ml} c_{ln} w^{-n} \Big) \bigg]\bigg|_{\rho=\rho_0}
\end{align*}
Hence, \eqnref{BC:equi1} is equivalent to
\beq\label{realpartrelation1}
\Re\bigg[\sum_{n=1}^\infty \sum_{m=1}^\infty n\Big( \sigma_m \alpha_m c_{mn} + \sigma_m s_{mn} -  \sum_{l=1}^\infty  \sigma_c \beta_{ml} c_{ln} \Big) w^{-n}+ \sum_{n=1}^\infty \sum_{m=1}^\infty n  \left(\sigma_c \beta_{mn} - \sigma_m \alpha_n \delta_{mn}\right) w^n \bigg]=0.
\eeq
It follows from Lemma \ref{Recoeff} that for each $n$,
\beq\label{someeq}
 \sum_{m=1}^\infty \sigma_m \left(\alpha_m c_{mn} + s_{mn} \right)- \sum_{m=1}^\infty\sum_{l=1}^\infty \sigma_c \beta_{ml} c_{ln} + \sum_{m=1}^\infty\overline{\left(\sigma_c \beta_{mn} - \sigma_m \alpha_n \delta_{mn}\right)} \, \gamma^{2n} = 0.
\eeq
This result completes the proof. 
\end{proof}

\begin{cor}
The condition $\sigma_m\pd{u}{\nu}\left|^+=\sigma_c \pd{u}{\nu}\right|^-$ is equivalent to 
\begin{align}
\boldsymbol{\beta}
&= \frac{\sigma_m}{\sigma_c} \boldsymbol{\alpha} - \frac{\sigma_m}{\sigma_c} \left(\overline{\boldsymbol{s}} +  \boldsymbol{s}  \gamma^{-2\mathcal{N}} \overline{C} \right)\left( I -  \gamma^{-2\mathcal{N}} C \gamma^{-2\mathcal{N}} \overline{C} \right)^{-1} \gamma^{-2\mathcal{N}},
\label{betaformula}
\end{align}
\end{cor}
\begin{proof}
Taking the complex conjugate and multiplying $C \gamma^{-2\mathcal{N}}$ to both sides of \eqnref{abcF1}, we obtain
$$
\sigma_c \boldsymbol{\beta}C \gamma^{-2\mathcal{N}} = \sigma_c \overline{\boldsymbol{\beta}  C}  \gamma^{-2\mathcal{N}}C \gamma^{-2\mathcal{N}} + \sigma_m \overline{\left( -\boldsymbol{\alpha} C + \overline{ \boldsymbol{\alpha}} \gamma^{2\mathcal{N}} - \boldsymbol{s} \right)} \gamma^{-2\mathcal{N}}C \gamma^{-2\mathcal{N}},
$$
and substitute this formula into \eqnref{abcF1} to get
\begin{align*}
\sigma_c \overline{\boldsymbol{\beta}}\left(I- \overline{C}  \gamma^{-2\mathcal{N}}C  \gamma^{-2\mathcal{N}}\right)  = \sigma_m \overline{\boldsymbol{\alpha}} \left(I- \overline{C}  \gamma^{-2\mathcal{N}}C  \gamma^{-2\mathcal{N}}\right) - \sigma_m \boldsymbol{s} \gamma^{-2\mathcal{N}} - \sigma_m \overline{\boldsymbol{s}} \gamma^{-2\mathcal{N}}C \gamma^{-2\mathcal{N}} .
\end{align*}
This result completes the proof with $G$ given by \eqnref{Gnotation}.
For the invertibility of $I- \gamma^{-2\mathcal{N}} C \gamma^{-2\mathcal{N}} \overline{C}$, we refer the reader to \cite{Choi:2021:ASR}. 
\end{proof}

Now, we consider the imperfect interface condition,
\beq\label{bdycond2}
\begin{aligned}
p(\Psi(w)) \left(u(\Psi(w))\big|_{\rho\rightarrow\rho_0^+} - u(\Psi(w))\big|_{\rho\rightarrow\rho_0^-} \right) 
&= \sigma_m \pd{u}{\nu}\Big|_{\p D}^+
 = \frac{\sigma_m}{h(\rho_0,\theta)}\frac{\partial }{\partial \rho} u(\Psi(w))\Big|_{\rho\rightarrow\rho_0^+}.
\end{aligned}
\eeq
If holds from \eqnref{uint} that
\begin{align*}
u(\Psi(w))\big|_{\rho\rightarrow\rho_0^+} - u(\Psi(w))\big|_{\rho\rightarrow\rho_0^-} \notag
=&\,\Re\bigg[ \sum_{m=1}^\infty\sum_{n=1}^\infty (\alpha_n \delta_{mn} - \beta_{mn}) w^n\\
&\quad+  \sum_{m=1}^\infty\sum_{n=1}^\infty \Big(\alpha_m c_{mn} + s_{mn} - \sum_{l=1}^\infty \beta_{ml} c_{ln}  \Big)w^{-n} \bigg]
\end{align*}
and
\begin{align*}
\frac{\partial }{\partial \rho}u(\Psi(w))\Big|_{\rho\rightarrow\rho_0^+} &= \Re\bigg[   \sum_{m=1}^\infty m \alpha_m w^m - \sum_{m=1}^\infty \sum_{n=1}^\infty n (\alpha_m c_{mn} + s_{mn} )w^{-n}  \bigg].
\end{align*}
The boundary condition \eqnref{bdycond2} is hence equivalent to
\begin{align*}
&p(\Psi(w))\Re\bigg[  \sum_{m=1}^\infty\sum_{n=0}^\infty (\boldsymbol{\alpha} - \boldsymbol{\beta})_{mn} w^n+ \sum_{m=1}^\infty \sum_{n=1}^\infty \big(\boldsymbol{\alpha} C + \boldsymbol{s} -  \boldsymbol{\beta} C  \big)_{mn} w^{-n} \bigg] \\
=& \frac{\sigma_m}{h(\rho_0,\theta)} \Re\bigg[   \sum_{m=1}^\infty m \alpha_m w^m - \sum_{m=1}^\infty \sum_{n=1}^\infty n\big(\boldsymbol{\alpha} C + \boldsymbol{s}  \big)_{mn} w^{-n}  \bigg] \quad \mbox{for } |w|=\gamma.
\end{align*}
Here and in the following, $[\cdot]_{mn}$ denotes the $(m,n)$--element of a matrix.
By using \eqnref{abcF1} and the fact that $h(\rho_0,\theta)p(w)$ is real-valued, we obtain
\beq\label{abc1}
\begin{aligned}
&\Re\bigg[ h(\rho_0,\theta) \, p(\Psi(w)) \Big( \sum_{m=1}^\infty\sum_{n=0}^\infty x_{mn} w^n+ \sum_{m=1}^\infty \sum_{n=1}^\infty y_{mn} w^{-n} \Big) \bigg]\\
=&\sigma_m \Re\bigg[   \sum_{m=1}^\infty m \alpha_m w^m - \sum_{m=1}^\infty \sum_{n=1}^\infty n\big(\boldsymbol{\alpha} C + \boldsymbol{s} \big)_{mn} w^{-n}  \bigg] \quad \mbox{for } |w|=\gamma,
\end{aligned}
\eeq
where
\beq\label{x_y}\begin{aligned}
 X &= \boldsymbol{\alpha} - \boldsymbol{\beta}, \quad
Y =  \boldsymbol{\alpha} C + \boldsymbol{s}  - \boldsymbol{\beta} C. \end{aligned}
\eeq

\subsection{Derivation of matrix formulas for the FPTs}
For notational simplicity, we set
$p(\theta)=p(\Psi(w))$ for $w=e^{\rho_0+i\theta}.$
We assume that $h(\rho_0,\theta) \, p(\theta)$ admits a Laurent series expansion: 
\beq\label{hp}
h(\rho_0,\theta) \, p(\theta) = \sum_{n=-\infty}^\infty p_n w^n\quad\mbox{ with }|w|=\gamma,
\eeq
where $p_n$'s are complex-valued coefficients.  In the numerical computation, we seek for $p$ with a finite number of nonzero coefficients $p_n$. 
Since $h(\rho_0,\theta) \, p(\theta)$ is real-valued, we have $\Re[ih(\rho_0,\theta) \, p(\theta)] =0$ and, by Lemma \ref{Recoeff}, $i p_{-n} + \overline{i p_n} \gamma^{2n}=0$ for all $n\in\ZZ$. It follows that
\beq\label{pp-}
p_{-n} = \overline{p_n} \,\gamma^{2n}\quad\mbox{for each }n\in\ZZ.
\eeq
By expanding the left hand side of \eqnref{abc1} and applying \eqnref{abc1}, we derive
\begin{align} \notag
&h(\rho_0,\theta) \, p(\theta) \Big( \sum_{m=1}^\infty\sum_{n=0}^\infty x_{mn} w^n+ \sum_{m=1}^\infty \sum_{n=0}^\infty y_{mn} w^{-n} \Big)\\
=\ &\sum_{m=1}^\infty\Big( \sum_{l=-\infty}^\infty p_l w^l \Big) \Big(\sum_{n=0}^\infty x_{mn} w^n+ \sum_{n=0}^\infty y_{mn} w^{-n} \Big)\notag\\
=\ &\sum_{m=1}^\infty \sum_{n=-\infty}^\infty \sum_{l=0}^\infty \left(p_{n-l}x_{ml} + p_{n+l} y_{ml}\right) w^n\notag
\end{align}
and, for $|w|=\gamma$,
\begin{align*}
\Re\bigg[ \sum_{m=1}^\infty \sum_{n=-\infty}^\infty \sum_{l=0}^\infty \left(p_{n-l}x_{ml} + p_{n+l} y_{ml}\right) w^n \bigg]
= \sigma_m \Re\bigg[  \sum_{m=1}^\infty  m \alpha_m w^m - \sum_{m=1}^\infty \sum_{n=1}^\infty n\big(\boldsymbol{\alpha} C +\boldsymbol{s} \big)_{mn} w^{-n}  \bigg].
\end{align*}
It holds from Lemma \ref{Recoeff} that, for $n\ge 0$,
\begin{align}
\sum_{l=0}^\infty \Big( p_{-n-l}x_{ml} + p_{-n+l} y_{ml} + \overline{(p_{n-l}x_{ml} + p_{n+l} y_{ml})}\gamma^{2n} \Big) 
&= \sigma_m \left[( \overline{\boldsymbol{\alpha}} \gamma^{2\NN} - \boldsymbol{\alpha} C - \boldsymbol{s} )\mathcal{N}\right]_{mn},\label{p1}
\end{align}
where $\mathcal{N}$ is given by \eqnref{Gnotation}, that is, the diagonal matrix with $(n,n)$-element $n$. 
Using \eqnref{pp-} and \eqnref{p1}, we obtain that for $m,n\ge 0$,
\begin{align}
\sigma_m \big[\left( \overline{\boldsymbol{\alpha}} \gamma^{2\mathcal{N}} - \boldsymbol{\alpha} C - \boldsymbol{s} \right)\mathcal{N}\big]_{mn}
&=  \sum_{l=0}^\infty \Big( x_{ml} \gamma^{2l} \overline{p_{l+n}} \gamma^{2n} +  \overline{y_{ml} p_{l+n}} \gamma^{2n}  + \overline{x_{ml}} \gamma^{2l} p_{l-n}   + y_{ml} p_{l-n}\Big)\notag\\
&= \sum_{l=0}^\infty \left[ \overline{( \overline{x_{ml}} \gamma^{2l}  + y_{ml})} \overline{p_{l+n}} \gamma^{2n} +( \overline{x_{ml}} \gamma^{2l}  + y_{ml} ) p_{l-n} \right],\label{abpxy}
\end{align}

We set
\beq\label{Pnotation}
\begin{aligned}
P^+ &= \big\{p_{mn}^+ \big\}_{m,n\ge 1}, \quad p_{mn}^+ = p_{m+n} \gamma^{m+n},\\
P^- &= \big\{p_{mn}^-\big\}_{m,n\ge 1},\quad   p_{mn}^- = p_{m-n} \gamma^{m-n}.
\end{aligned}
\eeq
From \eqnref{pp-}, we have $p_{nm}^- = p_{n-m} \gamma^{n-m} = \overline{p_{m-n}} \gamma^{m-n} = \overline{p_{mn}^-}$. In particular, $P^+$ is a Hankel matrix and $P^-$ is a Hermitian Toeplitz matrix (see Figure \ref{fig:PpPm}). 
\begin{figure}[h!]
\centering
\begin{subfigure}[t]{0.45\textwidth}
\centering
\be
\left[
\begin{array}{cccccc}
p_{2} \gamma^{2} & p_{3} \gamma^{3} & p_{4} \gamma^{4} & p_{5} \gamma^{5} & p_{6} \gamma^{6} & \cdots\\[3mm]
p_{3} \gamma^{3} & p_{4} \gamma^{4} & p_{5} \gamma^{5} & p_{6} \gamma^{6} & p_{7} \gamma^{7} & \cdots\\[3mm]
p_{4} \gamma^{4} & p_{5} \gamma^{5} & p_{6} \gamma^{6} & p_{7} \gamma^{7} & p_{8} \gamma^{8} & \cdots\\[3mm]
p_{5} \gamma^{5} & p_{6} \gamma^{6} & p_{7} \gamma^{7} & p_{8} \gamma^{8} & p_{9} \gamma^{9} & \cdots\\[3mm]
p_{6} \gamma^{6} & p_{7} \gamma^{7} & p_{8} \gamma^{8} & p_{9} \gamma^{9} & p_{10} \gamma^{10} & \cdots\\[3mm]
\vdots & \vdots & \vdots & \vdots & \vdots & \ddots 
\end{array}
\right]
\ee
\end{subfigure}
\begin{subfigure}[t]{0.45\textwidth}
\centering
\be
\left[
\begin{array}{cccccc}
p_{0} & \overline{p_{1}} \gamma^{1} & \overline{p_{2}} \gamma^{2} & \overline{p_{3}} \gamma^{3} & \overline{p_{4}} \gamma^{4} & \cdots\\[3mm]
p_{1} \gamma^{1} & p_{0} & \overline{p_{1}} \gamma^{1} & \overline{p_{2}} \gamma^{2} & \overline{p_{3}} \gamma^{3} & \cdots\\[3mm]
p_{2} \gamma^{2} & p_{1} \gamma^{1} & p_{0} & \overline{p_{1}} \gamma^{1} & \overline{p_{2}} \gamma^{2} & \cdots\\[3mm]
p_{3} \gamma^{3} & p_{2} \gamma^{2} & p_{1} \gamma^{1} & p_{0} & \overline{p_{1}} \gamma^{1} & \cdots\\[3mm]
p_{4} \gamma^{4} & p_{3} \gamma^{3} & p_{2} \gamma^{2} & p_{1} \gamma^{1} & p_{0} & \cdots\\[3mm]
\vdots & \vdots & \vdots & \vdots & \vdots & \ddots 
\end{array}
\right]
\ee
\end{subfigure}
\caption{Hankel matrix $P^+$ (left) and Hermitian Toeplitz  matrix $P^-$(right)}\label{fig:PpPm}
\end{figure}

We express \eqnref{abpxy} in a matrix form as
\begin{align}\label{Formula1}
\sigma_m \left( \overline{\boldsymbol{\alpha}} \gamma^{2\mathcal{N}} - \boldsymbol{\alpha} C - \boldsymbol{s} \right)\mathcal{N}
= \overline{\left( \overline{X} \gamma^{2\mathcal{N}} +Y \right)} \gamma^{-\mathcal{N}} \overline{P^{+}} \gamma^{\mathcal{N}} + \left(\overline{X} \gamma^{2\mathcal{N}}  +Y \right) \gamma^{-\mathcal{N}} P^{-} \gamma^{\mathcal{N}}.
\end{align}
The relations \eqnref{betaformula} and \eqnref{x_y} imply 
\begin{align}
&\sigma_c\left(\overline{X} \gamma^{2\mathcal{N}}  + Y\right) \notag \\
&= \sigma_c\left(\overline{\boldsymbol{\alpha}} \gamma^{2\mathcal{N}}  +  \boldsymbol{\alpha} C - \overline{\boldsymbol{\beta}}  \gamma^{2\mathcal{N}} - \boldsymbol{\beta} C + \boldsymbol{s}\right)  \notag\\
&= \overline{\boldsymbol{\alpha}} (\sigma_c-\sigma_m) \gamma^{2\mathcal{N}} + \boldsymbol{\alpha} (\sigma_c-\sigma_m) C \notag\\
&\quad + \boldsymbol{s} \left[(\sigma_c-\sigma_m)I + 2 \sigma_m ( I - \gamma^{-2\mathcal{N}} \overline{C} \gamma^{-2\mathcal{N}} C )^{-1}\right] + 2 \sigma_m \overline{\boldsymbol{s}} \left[ ( I -  \gamma^{-2\mathcal{N}} C \gamma^{-2\mathcal{N}} \overline{C})^{-1}  \gamma^{-2\mathcal{N}} C \right].\label{Formula2}
\end{align}
Solving \eqnref{Formula1} and \eqnref{Formula2}, we obtain 
\begin{align}
\boldsymbol{\alpha} A_1  + \overline{\boldsymbol{\alpha}} \overline{A_2}  + \overline{\boldsymbol{s}} \overline{B_1} + \boldsymbol{s} B_2 = O,\label{aBcF}
\end{align}
where $O$ means the zero matrix and
\beq\label{ABMN}
\begin{cases}
\ds A_1 &= (\sigma_c-\sigma_m) \gamma^{\mathcal{N}}\, \overline{P^{+}} \,\gamma^{\mathcal{N}} + (\sigma_c-\sigma_m) C \gamma^{-\mathcal{N}} P^{-} \gamma^{\mathcal{N}} + \sigma_c \sigma_m C \mathcal{N}, \\[2mm]
\ds A_2 &= (\sigma_c-\sigma_m) \gamma^{\mathcal{N}} \,\overline{P^{-}}\, \gamma^{\mathcal{N}} + (\sigma_c-\sigma_m) C \gamma^{-\mathcal{N}} P^{+} \gamma^{\mathcal{N}} - \sigma_c \sigma_m \gamma^{2\mathcal{N}} \mathcal{N}, \\[2mm]
\ds B_1 &= \left[ (\sigma_c-\sigma_m) I + 2\sigma_m( I - \gamma^{-2\mathcal{N}}\overline{C} \gamma^{-2\mathcal{N}} C )^{-1} \right] \gamma^{-\mathcal{N}} P^{+} \gamma^{\mathcal{N}} \\
&\quad + 2\sigma_m( I -\gamma^{-2\mathcal{N}} \overline{C} \gamma^{-2\mathcal{N}} C )^{-1} \,\gamma^{-2\mathcal{N}} \overline{C} \,\gamma^{-\mathcal{N}} \,\overline{P^{-}} \,\gamma^{\mathcal{N}},\\[2mm]
\ds B_2 &= \left[ (\sigma_c-\sigma_m) I + 2\sigma_m( I -\gamma^{-2\mathcal{N}} \overline{C} \gamma^{-2\mathcal{N}} C )^{-1} \right] \gamma^{-\mathcal{N}} P^{-} \gamma^{\mathcal{N}} \\
&\quad + 2\sigma_m(I - \gamma^{-2\mathcal{N}} \overline{C} \gamma^{-2\mathcal{N}} C)^{-1} \, \gamma^{-2\mathcal{N}} \overline{C} \,\gamma^{-\mathcal{N}}\, \overline{P^{+}} \,\gamma^{\mathcal{N}} + \sigma_c \sigma_m \,\mathcal{N}.
\end{cases}
\eeq

If $B_2$ is invertible, \eqnref{aBcF} leads us to (by taking the complex conjugate)
$$
\overline{\boldsymbol{s}} = -\overline{\boldsymbol{\alpha}}\, \overline{A_1}\, \overline{B_2}^{-1}  - \boldsymbol{\alpha} A_2 \overline{B_2}^{-1} - \boldsymbol{s} B_1 \overline{B_2}^{-1}.
$$
Assuming further the invertibility for $B_2 - B_1 B_2^{-1} \overline{B_1}$ and substituting the above relation into \eqnref{aBcF}, we finally arrive to
\beq\label{alpha:expression}
\begin{aligned}
\boldsymbol{s} = &-\boldsymbol{\alpha} \left[\left( A_1 - A_2 \overline{B_2}^{-1} \overline{B_1} \right) \left(B_2 - B_1 \overline{B_2}^{-1} \overline{B_1}\right)^{-1} \right]\\
& - \overline{\boldsymbol{\alpha}} \left[\left( \overline{A_2} - \overline{A_1}\, \overline{B_2}^{-1} \overline{B_1} \right) \left(B_2 - B_1 \overline{B_2}^{-1} \overline{B_1} \right)^{-1}\right].
\end{aligned}
\eeq
Recall the FPTs $\FF_{mn}^{(1)}$ and $\FF_{mn}^{(2)}$ are defined by \eqnref{s}. 
From \eqnref{alpha:expression}, we obtain the matrix expressions for the FPTs.
\begin{lemma}\label{FPTmatrix}
Let $A_1,A_2,B_1,B_2$ are given by \eqnref{ABMN} (see also \eqnref{Gnotation} and \eqnref{Pnotation}). Then, FPTs of the interface problem \eqnref{LCproblem} have matrix formulations,
\begin{align*}
&\FF^{(1)} = 4\pi \left( A_1 - A_2 \overline{B_2}^{-1} \overline{B_1} \right) \left(B_2 - B_1 \overline{B_2}^{-1} \overline{B_1}\right)^{-1} \mathcal{N},\\
&\FF^{(2)} = 4\pi \left( \overline{A_2} - \overline{A_1}\, \overline{B_2}^{-1} \overline{B_1} \right) \left(B_2 - B_1 \overline{B_2}^{-1} \overline{B_1} \right)^{-1} \mathcal{N},
\end{align*}
if $B_2$ and $B_2 - B_1 B_2^{-1} \overline{B_1}$ are invertible. 
\end{lemma}

For the numerical computation in Section \ref{sec:numerical}, we use the finite truncations of the matrices $A_1,A_2,B_1,B_2$ with a large matrix dimension. For all examples in Section  \ref{sec:numerical}, the finite truncations of $B_2$ and $B_2 - B_1 B_2^{-1} \overline{B_1}$ are invertible. 

As a corollary of Lemma \ref{FPTmatrix} and the formula \eqnref{NQFQ}, we explicitly obtain decompositions of the CGPTs as follows.
\begin{theorem}\label{Factorization:smooth}
The CGPTs associated with the imperfect interface problem \eqnref{LCproblem} admit the matrix decompositions
\begin{align*}
 \mathbb{N}^{(1)} & = 4\pi \, Q^{-1} \left( A_1 - A_2 \overline{B_2}^{-1} \overline{B_1} \right) \left(B_2 - B_1 \overline{B_2}^{-1} \overline{B_1}\right)^{-1} \mathcal{N} \left(Q^{-1}\right)^T, \\
 \mathbb{N}^{(2)} &= 4\pi \, \overline{Q^{-1}} \left( \overline{A_2} - \overline{A_1}\, \overline{B_2}^{-1} \overline{B_1} \right) \left(B_2 - B_1 \overline{B_2}^{-1} \overline{B_1} \right)^{-1} \mathcal{N} \left( Q^{-1} \right)^T,
\end{align*}
where $Q=(q_{mn})_{m,n=1}^\infty$ is the lower triangular matrix of Faber polynomial coefficients in \eqnref{Fm:pmn}.
\end{theorem}

\subsection{GPTs for a disk}

If $D$ is a disk with the radius $\gamma$, then $C$ is identically the zero matrix and $Q$ is the identity matrix. Consequently, \eqnref{ABMN} leads
\be
\begin{cases}
\ds A_1 &= (\sigma_c-\sigma_m) \gamma^{\mathcal{N}}\, \overline{P^{+}} \,\gamma^{\mathcal{N}}, \\[2mm]
\ds A_2 &= (\sigma_c-\sigma_m) \gamma^{\mathcal{N}} \,\overline{P^{-}}\, \gamma^{\mathcal{N}} - \sigma_c \sigma_m \gamma^{2\mathcal{N}} \mathcal{N},\\[2mm]
\ds B_1 &= (\sigma_c-\sigma_m) \gamma^{-\mathcal{N}} P^{+} \gamma^{\mathcal{N}},\\[2mm]
\ds B_2 &= (\sigma_c-\sigma_m) \gamma^{-\mathcal{N}} P^{-} \gamma^{\mathcal{N}} + \sigma_c \sigma_m \,\mathcal{N}, 
\end{cases}
\ee
which gives $A_1  = \gamma^{2\mathcal{N}} \overline{B_1}$ and $A_2  = \gamma^{2\mathcal{N}} \overline{B_2} - 2\sigma_c \sigma_m \gamma^{2\mathcal{N}} \mathcal{N}$.  It then follows from Theorem \ref{Factorization:smooth} that
\begin{align*}
&\NN_{mn}^{(1)} = 8\pi mn \sigma_c \sigma_m \gamma^{2m}\left[ \overline{B_2}^{-1} \overline{B_1} \left(B_2 - B_1 \overline{B_2}^{-1} \overline{B_1}\right)^{-1}  \right]_{mn},\\
&\NN_{mn}^{(2)} = 4\pi n \gamma^{2m}\delta_{mn}  - 8\pi mn \sigma_c \sigma_m \gamma^{2m} \left[\left(B_2 - B_1 \overline{B_2}^{-1} \overline{B_1} \right)^{-1} \right]_{mn}.
\end{align*}
If we set $p_n=0$ for all $n\neq0$, then $P^+=O$ and $P^- = p_0I$, and thus, 
$$B_1=O,\quad B_2 = (\sigma_c-\sigma_m)p_0I +\sigma_c \sigma_m \mathcal{N}.$$
We then get the following corollary.
\begin{cor}
For a disk $D$,  if $p_n=0$ for all $n\neq0$,  namely,  if the interface function is constant, then GPTs of the interface problem \eqnref{LCproblem} are
\begin{align*}
\NN_{mn}^{(1)} = 0, \quad \NN_{mn}^{(2)} = 4\pi n \gamma^{2m}\delta_{mn}  \frac{(\sigma_c-\sigma_m)p_0 - \sigma_c \sigma_m m}{(\sigma_c-\sigma_m)p_0  +\sigma_c \sigma_m m}.
\end{align*}
\end{cor}
By this corollary,  $\NN_{1n}^{(1)} = \NN_{1n}^{(2)} = 0$ for all $n\ge1$ if
$$
p_0 =  \frac{\sigma_c \sigma_m}{\sigma_c - \sigma_m}.
$$
In other words,  by using \eqnref{hp} and $h(\rho_0,\theta) = \gamma$ for the disk $\Omega$, if we set $p$ as 
$$
p = \frac{1}{\gamma} \frac{\sigma_c \sigma_m}{\sigma_c-\sigma_m}
$$
the inclusion is perfectly neutral and does not perturb the incident field at all when the incident field is unform \cite{Torquato:1995:EIP}.

\section{GPT-vanishing inclusions with imperfect interfaces}\label{sec:numerical}

The computation in this section is very similiar with the numerical part in \cite{Choi:2023:GME}. However,  main difference is that we iterate the coefficients of the interface function instead of conductivites of inclusions.
We utilize FPTs to construct imperfect interfaces that vanish GPTs. Specifically, we determine the interface condition function $p$ for a given pair $(D,\sigma_c)$ such that FPTs vanish up to a finite order. Then,  GPTs would also vanish up to the same finite order.

\subsection{Numerical scheme}\label{subsec:FPT_reducing}
For a given $D$, we use the quantities $\gamma$, $\Psi(w)$, and $(\rho,\theta)$ as defined in Section \ref{sec:Faber}. The interface function, $p$, and its coefficient matrices, $P^+$ and $P^-$, are defined by \eqnref{hp}, \eqnref{Pnotation}, and Figure \ref{fig:PpPm}. Let us assume that $\sigma_m=1$ and $\sigma_c$ is fixed. The incident field is assumed to be a uniform field, which means that we only need to consider FPTs for $m=1$.

Our objective is to find $\boldsymbol{p} = (p_0,\dots,p_{N-1})$ such that the leading terms of the FPTs vanish. To accomplish this, we aim to solve the following equation:
\beq\label{fsig:condition}
\boldsymbol{f}(\boldsymbol{p}) = \boldsymbol{0},
\eeq
where $\boldsymbol{f}:\RR^{N}\to\RR^{2N}$ is a nonlinear vector function mapping  $\boldsymbol{p}\longmapsto (f_1,...,f_{2N})$. Here, $f_n$ is given by:
\begin{align*}
f_n =
\begin{cases}
&\left|\dfrac{\FF^{(1)}_{1n}}{4\pi n}\right|,\quad 1\le n \le N,\\[5mm]
&\left|\dfrac{\FF^{(2)}_{1n}}{4\pi n}\right|,\quad N+1\le n \le 2N,
\end{cases}
\end{align*}
In order to achieve $\boldsymbol{p}$ with \eqnref{fsig:condition}, we iterate the multidimensional Newton method:
\begin{align}\label{initialguess}
\boldsymbol{p}^{(k+1)} = \boldsymbol{p}^{(k)} - \alpha\, \boldsymbol{J}^{\dagger}\big[\boldsymbol{p}^{(k)}\big]\, \boldsymbol{f}\big[\boldsymbol{p}^{(k)}\big],
\end{align}
where the learning rate is denoted by $\alpha$, while $k$ represents the iteration number. Additionally, $\boldsymbol{J}^\dagger$ refers to the pseudo-inverse of the Jacobian matrix of $\boldsymbol{f}$.  We start with seed input $\boldsymbol{p}^{(0)}=(p_0^{(0)},\dots,p_{N-1}^{(0)})$, with each coefficient having a magnitude less than 2.

We use Theorem \ref{FPTmatrix} to calculate $\boldsymbol{f}(\boldsymbol{p}^{(k)})$, truncating all semi-infinite matrices to finite matrices with size $100\times100$. To compute $\boldsymbol{J}^{\dagger}[\boldsymbol{p}^{(k)}]$, we apply the finite difference method to the partial derivatives of $\boldsymbol{f}$ based on Matlab R2022b. We continue applying the multidimensional Newton method until the coefficients of the interface function meet the stopping criterion:
$$
\frac{|\boldsymbol{p}^{(k+1)}-\boldsymbol{p}^{(k)}|}{|\boldsymbol{p}^{(k)}|} < 10^{-15}.
$$

\subsection{Examples}

We visualize our results on the construction of the GPT-vanishing structures. We plot the exterior potential in \eqnref{uext:ext} induced by the presence of inclusion based on \eqnref{s} and Theorem \ref{FPTmatrix}.

We give three different boundary conditions that are perfectly bonding as \eqnref{pinf} and imperfectly bonding as \eqnref{p02} and \eqnref{p0123}.
For perfectly bonding case, we set $p = \infty$, and hence, the boundary condition 
\be
\sigma_m \pd{u}{\nu}\Big|^+=\sigma_c \pd{u}{\nu}\Big|^-= p (u\big|^+ - u\big|^- )  \qquad \mbox{on }\p D
\ee
is equivalent to
\beq\label{pinf}
u\big|^+ = u\big|^- , \quad \sigma_m \pd{u}{\nu}\Big|^+=\sigma_c \pd{u}{\nu}\Big|^- \qquad \mbox{on }\p D,
\eeq
which are the typical transmission conditions.

For GPTs of order 1 vanishing case, namely, PT-vanishing structure, we give 
\beq\label{p02}
P(w) = p_0 + p_2 w^2 + \overline{p_2} \gamma^4 w^{-2}.
\eeq
Since $p_0$ and $p_2$ are free variables, we have the degree of freedom as 3. This is enough to vanish two quantities and we could make the weakly neutral inclusion satisfying $\NN_{11}^{(1)}=\NN_{11}^{(2)}\approx0$.

For GPTs of order up to 2 vanishing structure, we give 
\beq\label{p0123}
P(w) = p_0 + p_1 w + \overline{p_1} \gamma^2 w^{-1} +p_2 w^2 + \overline{p_2} \gamma^4 w^{-2} + p_3 w^3 + \overline{p_3} \gamma^6 w^{-3}.
\eeq
Now we have four free variables, which enable us to control four quantities that vanish. Thus,  provided that $P$ is given as \eqnref{p0123},  we can make the GPTs of order up to 2 vanishing structures satisfying $\NN_{11}^{(1)}=\NN_{11}^{(2)}=\NN_{12}^{(1)}=\NN_{12}^{(2)} \approx 0$. This is quite surprising because the solution $u$ in \eqnref{LCproblem} could decay faster than the weakly neutral inclusion as follows:
$$
u(x) = a\cdot x + O(|x|^{-3}),\quad |x|\to\infty.
$$

\begin{example} We first consider a domain with conformal mapping given by
$$
\Psi(w) = w + \frac{0.25}{w} + \frac{0.125}{w^2} + \frac{0.1}{w^3}.
$$
To construct the PT-vanishing structure, we set $P$ as defined in \eqnref{p02}, where $p_0 = 1.5925$ and $p_2 = -0.7240$. For the GPTs of order up to 2 vanishing structure, we define $P$ as in \eqnref{p0123}, with $p_0 = 1.7214$, $p_1 = 0.1933$, $p_2 = -0.7480$, and $p_3 = -0.5522$. The field perturbation among three different boundary conditions is shown in Figure \ref{weaklyneutral1}.
\end{example}
\begin{figure}[H]
\begin{subfigure}{\linewidth}
\centering
\captionsetup{justification=centering}
\begin{minipage}{0.33\linewidth}
\subcaption*{Perfectly bonding}
\end{minipage}\hspace*{1mm}
\begin{minipage}{0.33\linewidth}
\subcaption*{PT-vanishing}
\end{minipage}\hspace*{1mm}
\begin{minipage}{0.33\linewidth}
\subcaption*{GPT-vanishing up to order 2}
\end{minipage}\hspace*{1mm}
\end{subfigure}\\
\begin{subfigure}{\linewidth}
\centering
\begin{minipage}{0.33\linewidth}
\includegraphics[width=\linewidth, trim={33mm 84mm 34mm 83mm}, clip]{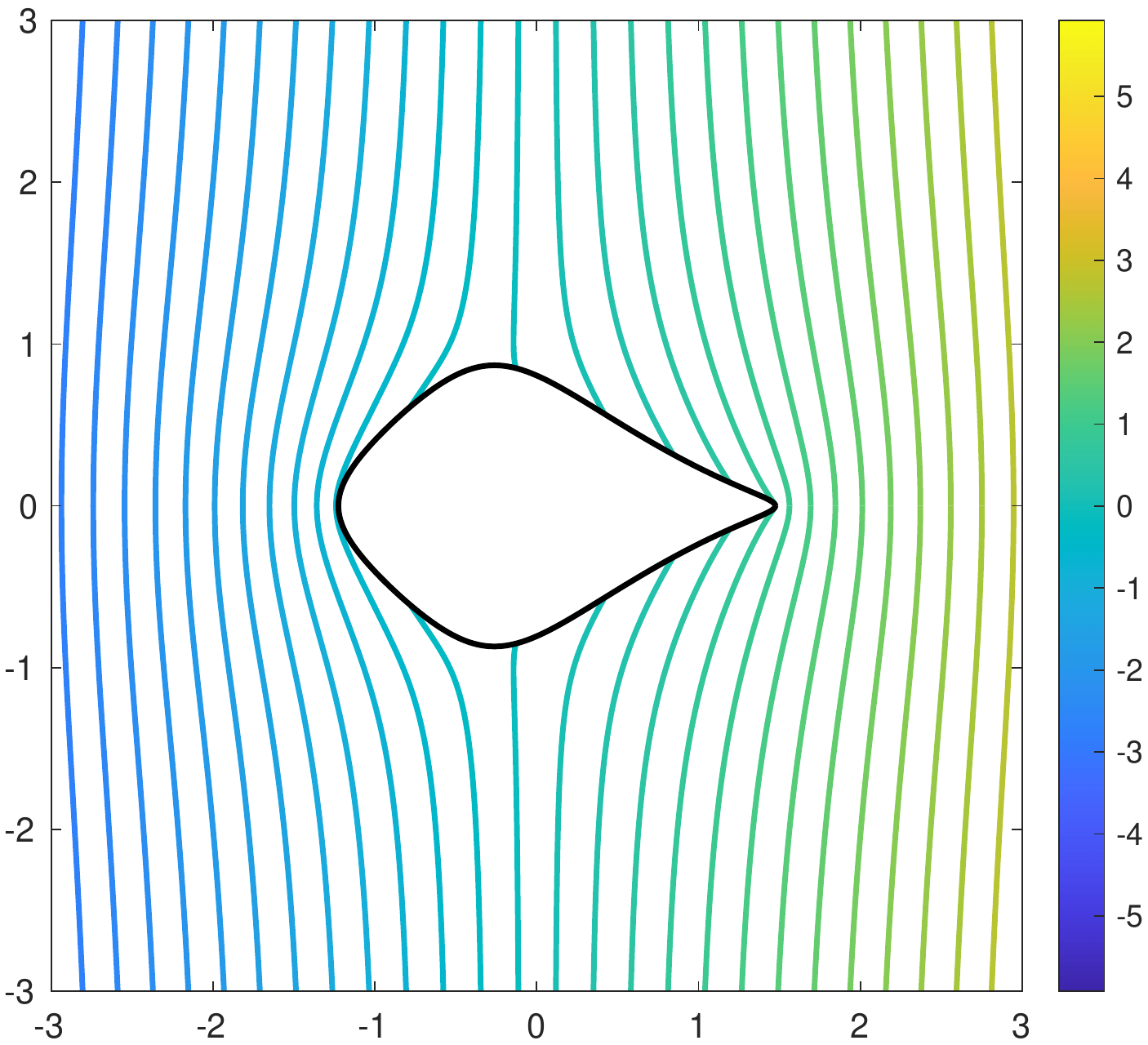}
\end{minipage}\hspace*{1mm}
\begin{minipage}{0.33\linewidth}
\includegraphics[width=\linewidth, trim={33mm 84mm 34mm 83mm}, clip]{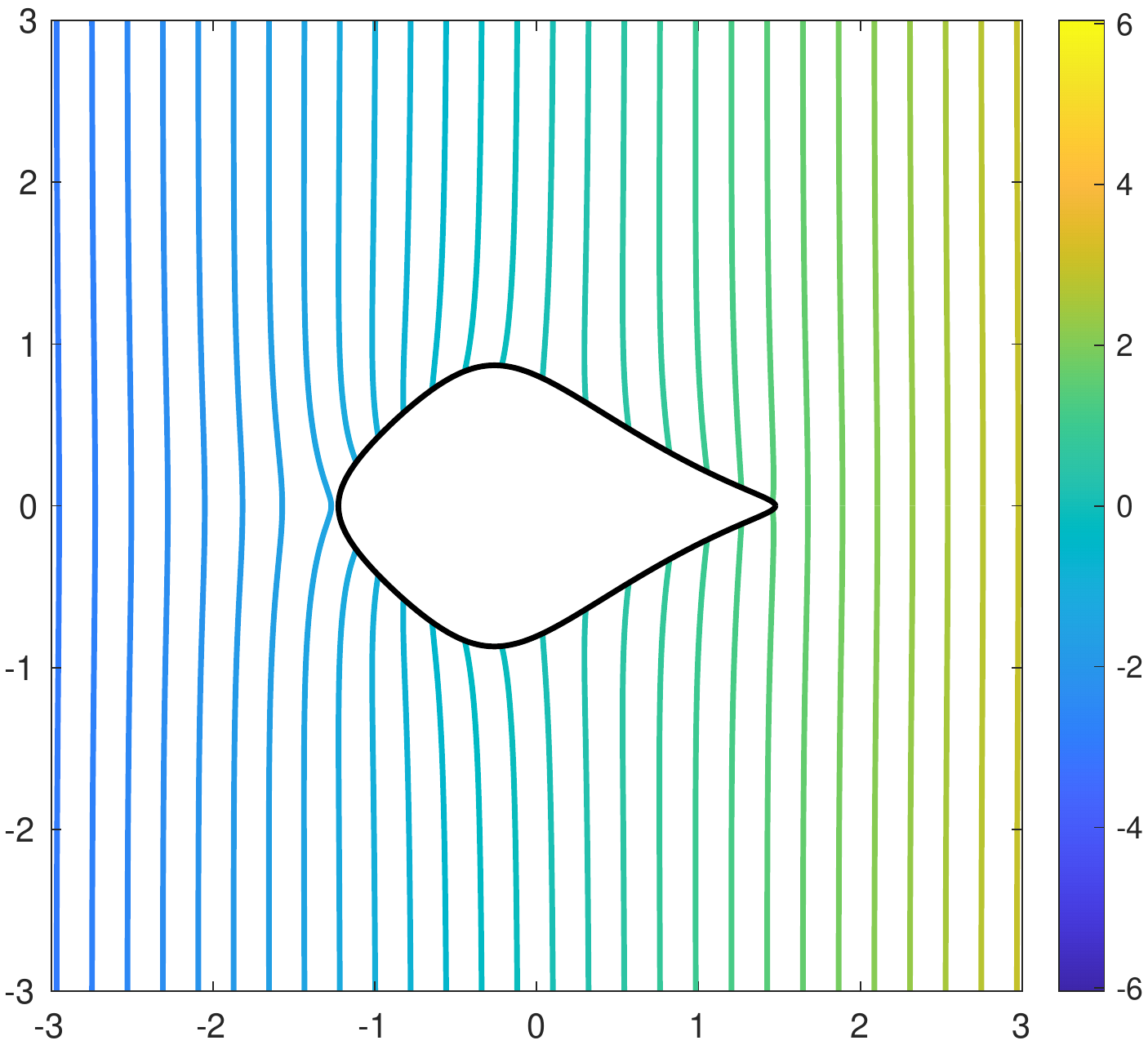}
\end{minipage}\hspace*{1mm}
\begin{minipage}{0.33\linewidth}
\includegraphics[width=\linewidth, trim={33mm 84mm 34mm 83mm}, clip]{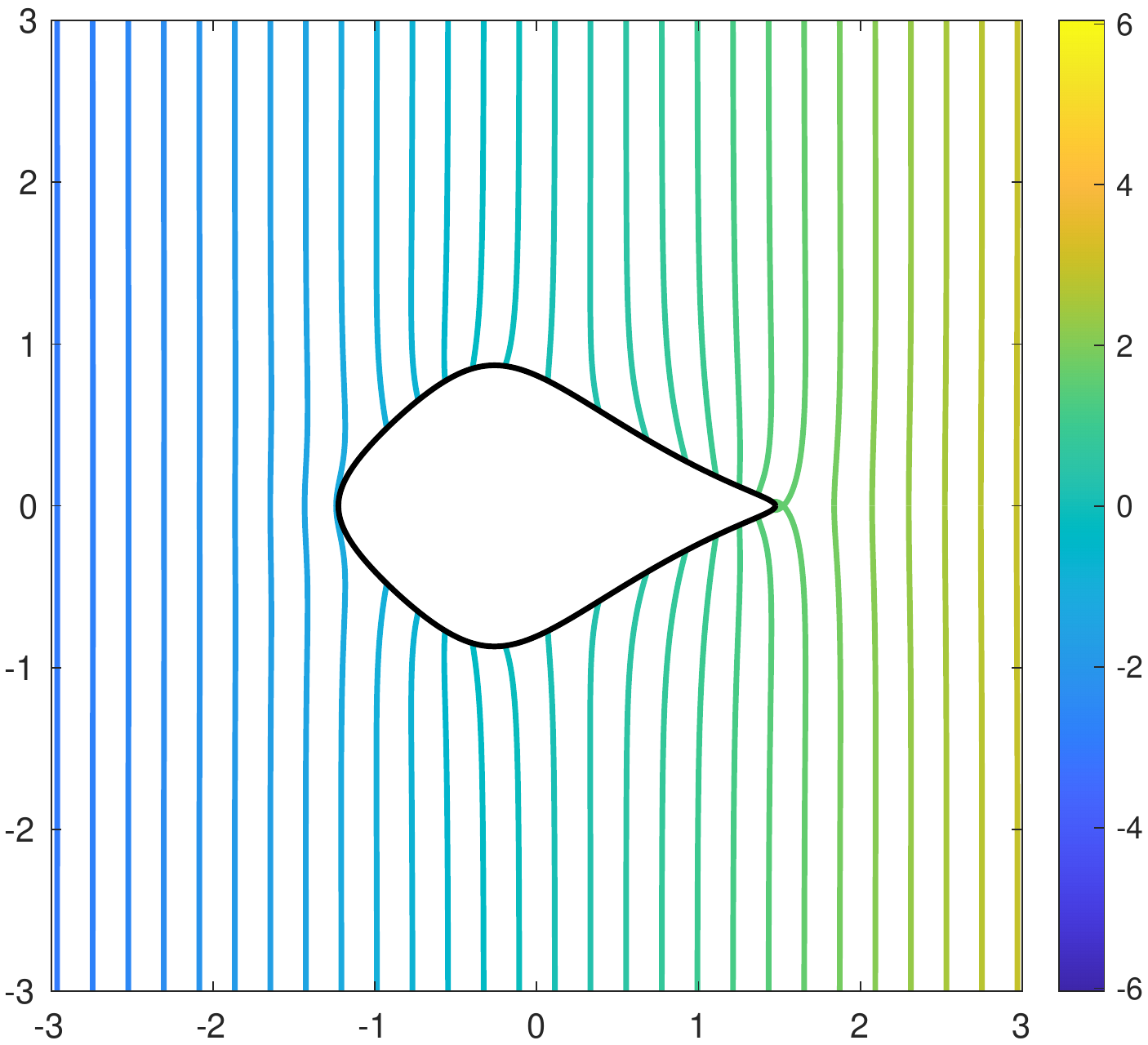}
\end{minipage}\hspace*{1mm}
\end{subfigure}\\[3mm]
\begin{subfigure}{\linewidth}
\centering
\begin{minipage}{0.33\linewidth}
\includegraphics[width=\linewidth, trim={33mm 84mm 34mm 83mm}, clip]{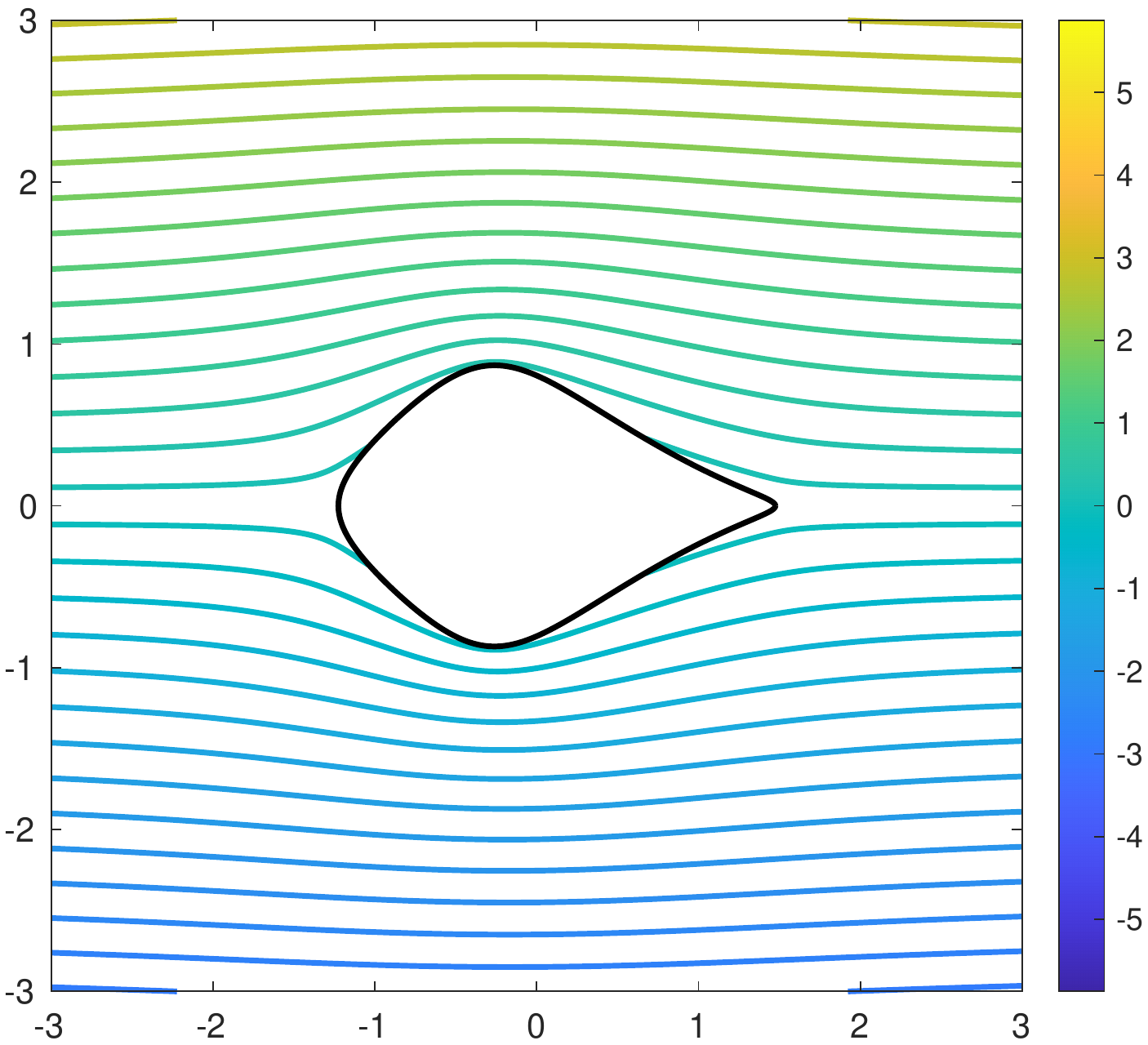}
\end{minipage}\hspace*{1mm}
\begin{minipage}{0.33\linewidth}
\includegraphics[width=\linewidth, trim={33mm 84mm 34mm 83mm}, clip]{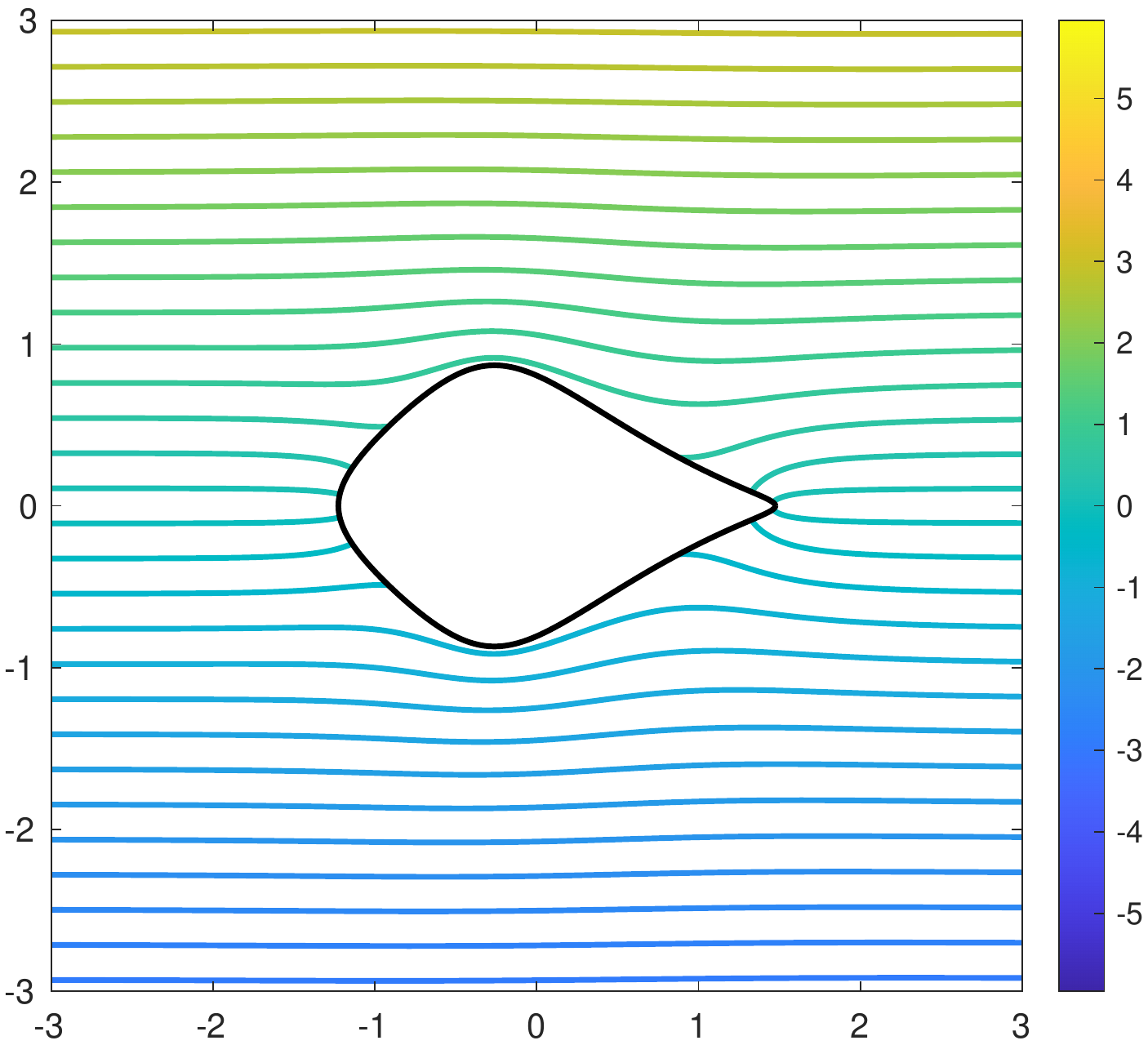}
\end{minipage}\hspace*{1mm}
\begin{minipage}{0.33\linewidth}
\includegraphics[width=\linewidth, trim={33mm 84mm 34mm 83mm}, clip]{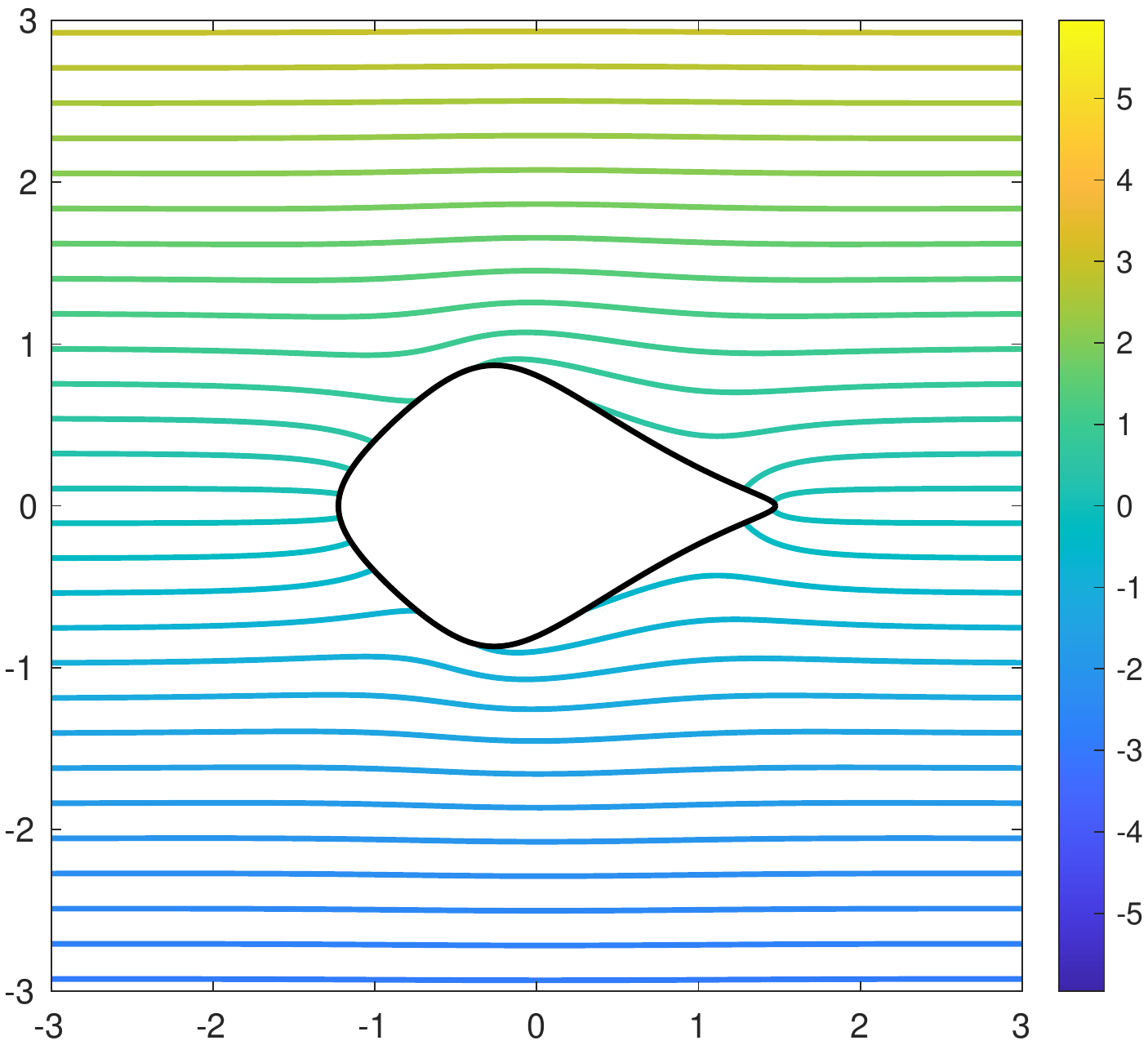}
\end{minipage}\hspace*{1mm}
\end{subfigure}
\caption{The background field is given by $H(x)=x_1$. Colored curves represent contours of the potential $u$. We set $\gamma = 1$ and the conductivity of the inclusion is given by $\sigma_c = 5$.}\label{weaklyneutral1}
\end{figure}

{\renewcommand{\arraystretch}{1.7}
\begin{table}[H]
\centering
\begin{tabular}{|c|c|c|c|c|}
\hline
& $\NN_{11}^{(1)}$ & $\NN_{12}^{(1)}$  & $\NN_{11}^{(2)}$ & $\NN_{12}^{(2)}$ \\
\hline\hline
Perfectly bonding & 1.1693 & 1.0918 & 7.7190 & $-0.7977$ \\
\hline
1st order vanishing & $-2.1799\times 10^{-16}$ & 2.6301 & $3.8367\times 10^{-15}$ & $-0.1007$ \\
\hline
2nd order vanishing & $-3.3353\times 10^{-15}$& $1.7439\times 10^{-16}$ & $2.6159\times 10^{-16}$ & $2.2846\times 10^{-14}$ \\
\hline
\end{tabular}
\caption{The leading GPTs for the domains in Figure \ref{weaklyneutral1}.}\label{table_sharp}
\end{table}
}

\begin{example} We now observe a kite-shaped domain with conformal mapping
$$
\Psi(w) = w + \frac{0.1}{w} + \frac{0.25}{w^2} - \frac{0.05}{w^3} + \frac{0.05}{w^4} - \frac{0.04}{w^5} + \frac{0.02}{w^6}.
$$
To construct the PT-vanishing structure, we selected values for $P$ based on equation \eqnref{p02}, where $p_0$ is set to $1.4022$ and $p_2$ is set to $-0.2742$. For GPTs up to order 2 vanishing structure, we determined the values for $P$ using equation \eqnref{p0123}, with $p_0$ set to $1.7975$, $p_1$ set to $0.1582$, $p_2$ set to $-0.3251$, and $p_3$ set to $-1.0352$. Figure \ref{weaklyneutral2} illustrates the field perturbation under three different boundary conditions.
\end{example}

\begin{figure}[H]
\begin{subfigure}{\linewidth}
\centering
\captionsetup{justification=centering}
\begin{minipage}{0.33\linewidth}
\subcaption*{Perfectly bonding}
\end{minipage}\hspace*{1mm}
\begin{minipage}{0.33\linewidth}
\subcaption*{PT-vanishing}
\end{minipage}\hspace*{1mm}
\begin{minipage}{0.33\linewidth}
\subcaption*{GPT-vanishing up to order 2}
\end{minipage}\hspace*{1mm}
\end{subfigure}\\
\begin{subfigure}{\linewidth}
\centering
\begin{minipage}{0.33\linewidth}
\includegraphics[width=\linewidth, trim={33mm 84mm 34mm 83mm}, clip]{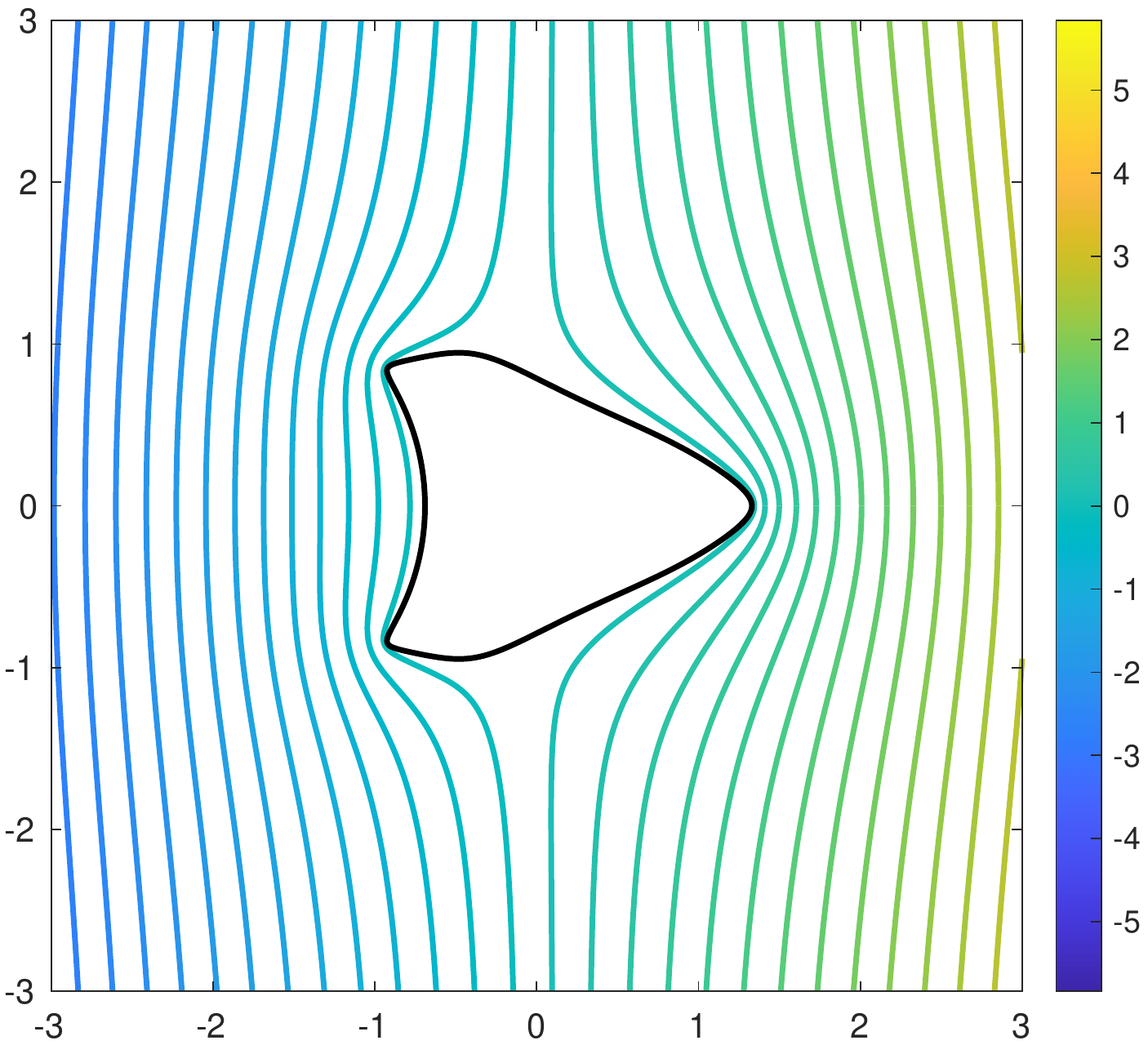}
\end{minipage}\hspace*{1mm}
\begin{minipage}{0.33\linewidth}
\includegraphics[width=\linewidth, trim={33mm 84mm 34mm 83mm}, clip]{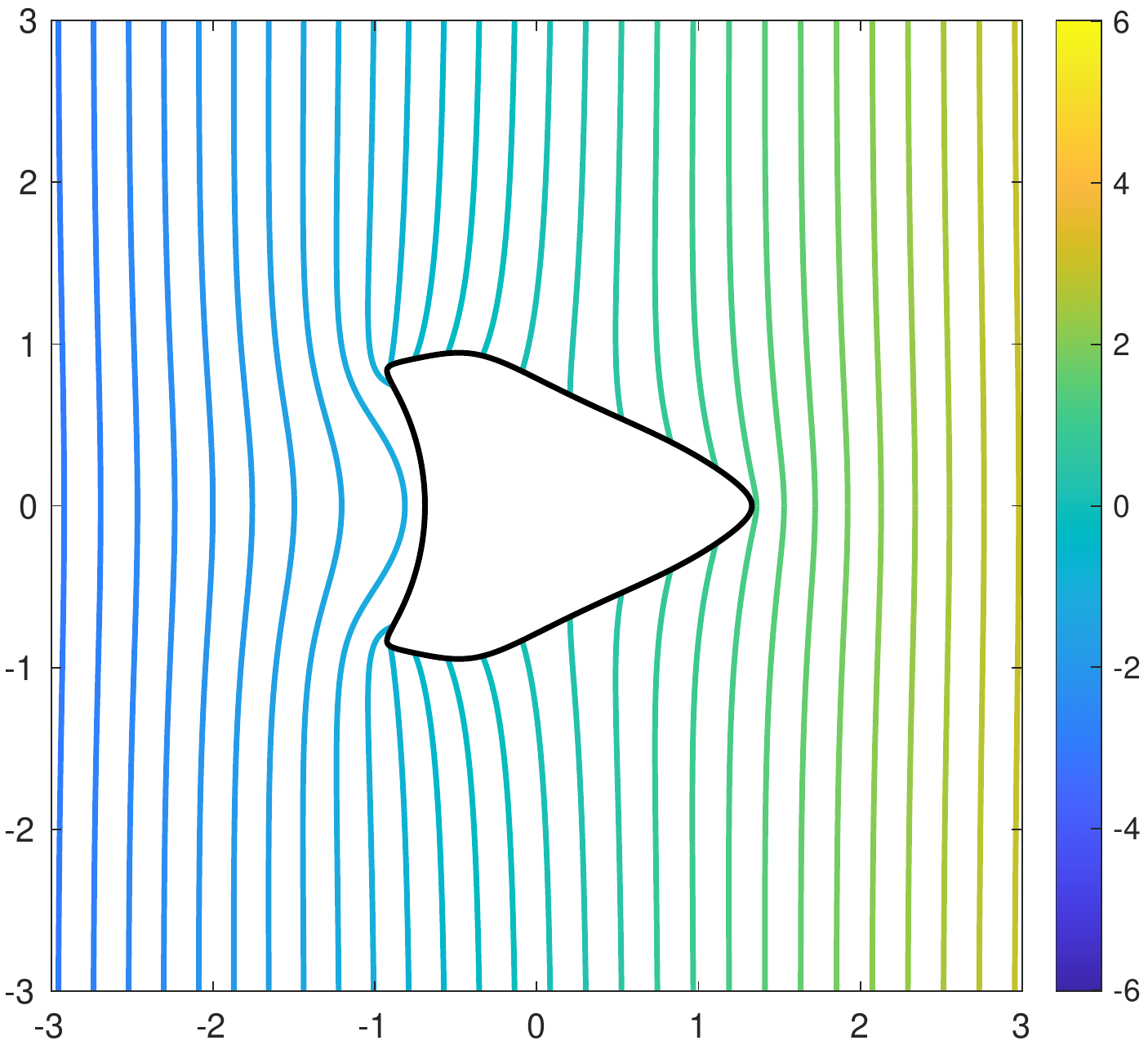}
\end{minipage}\hspace*{1mm}
\begin{minipage}{0.33\linewidth}
\includegraphics[width=\linewidth, trim={33mm 84mm 34mm 83mm}, clip]{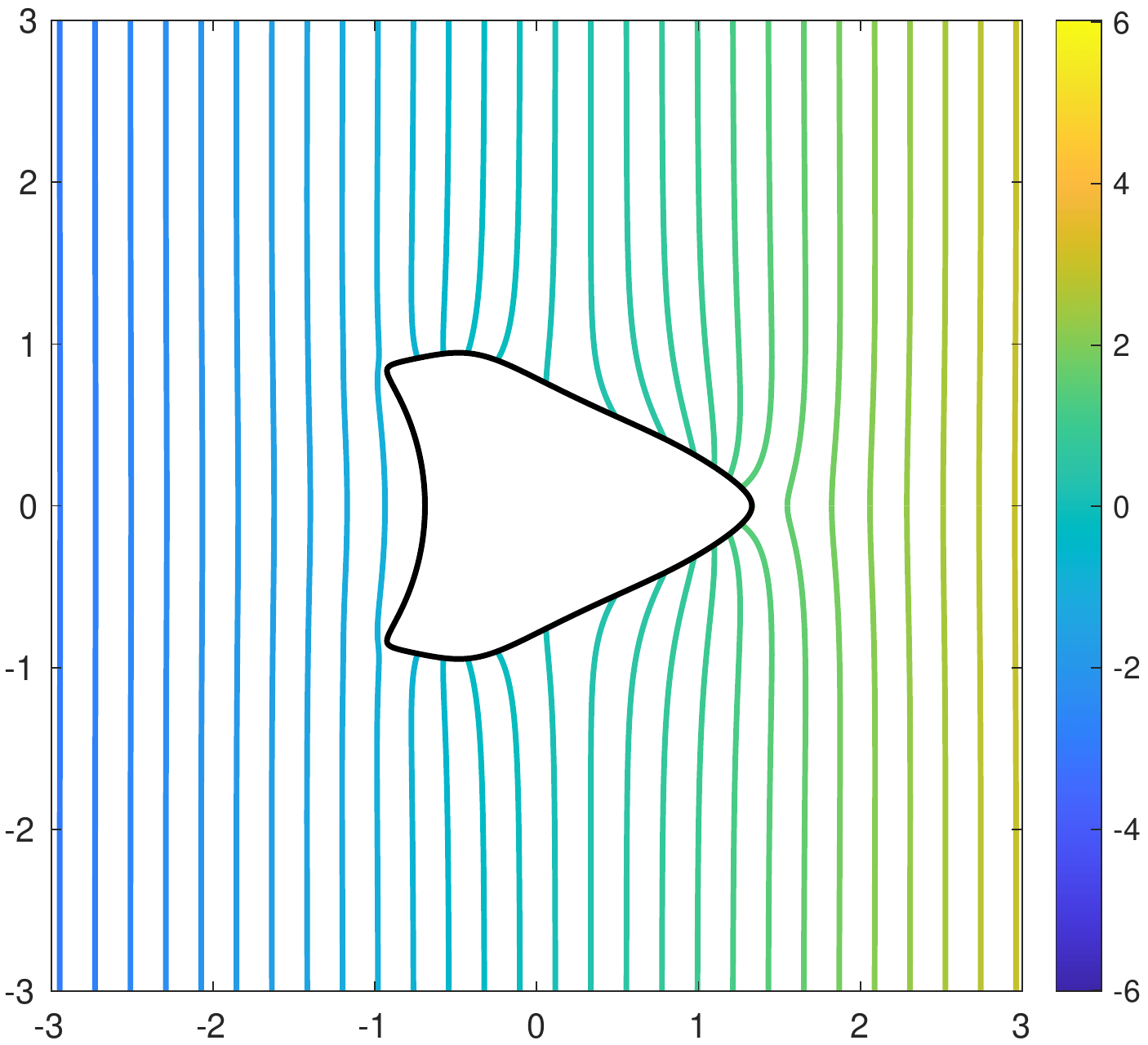}
\end{minipage}\hspace*{1mm}
\end{subfigure}\\[3mm]
\begin{subfigure}{\linewidth}
\centering
\begin{minipage}{0.33\linewidth}
\includegraphics[width=\linewidth, trim={33mm 84mm 34mm 83mm}, clip]{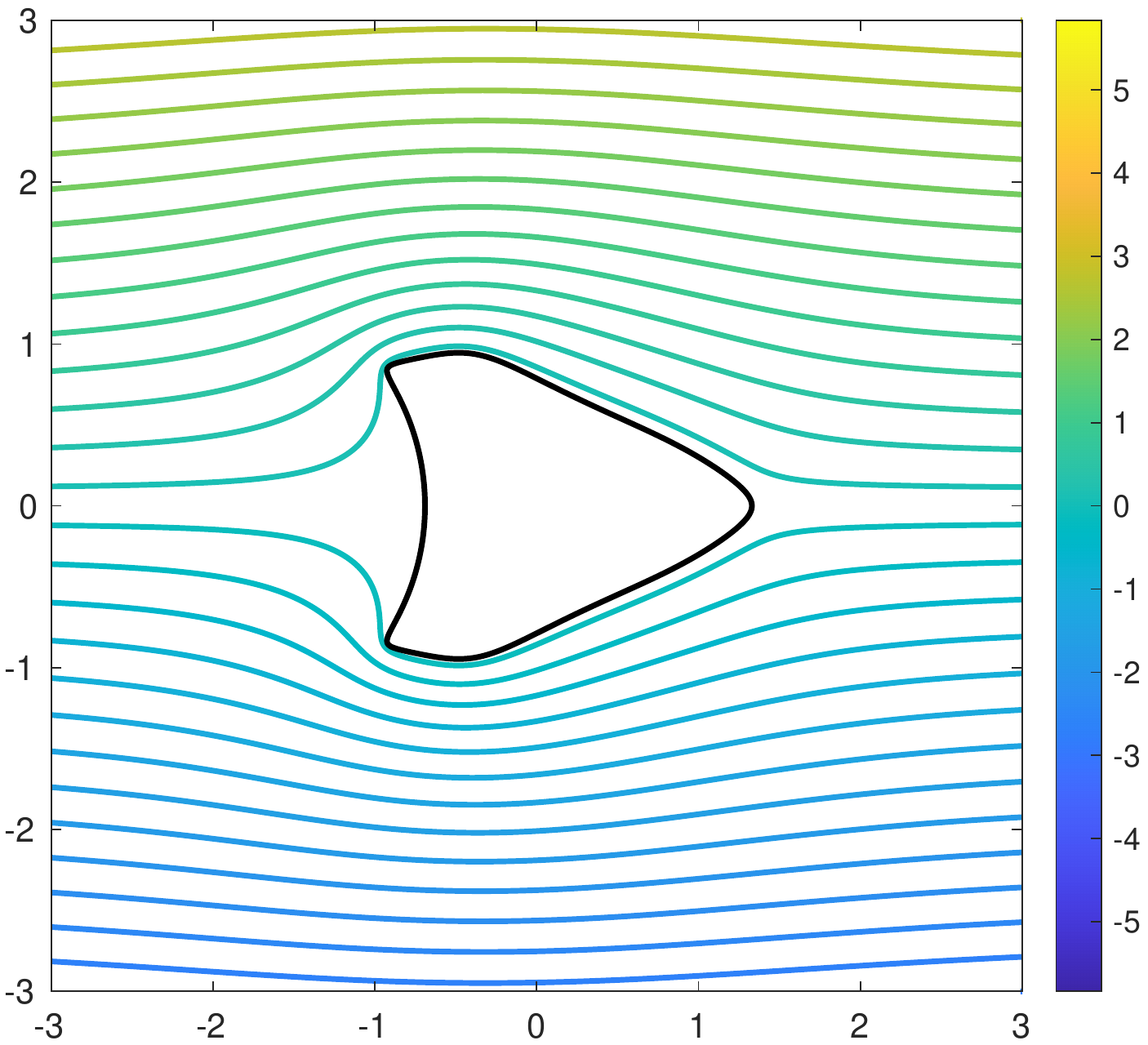}
\end{minipage}\hspace*{1mm}
\begin{minipage}{0.33\linewidth}
\includegraphics[width=\linewidth, trim={33mm 84mm 34mm 83mm}, clip]{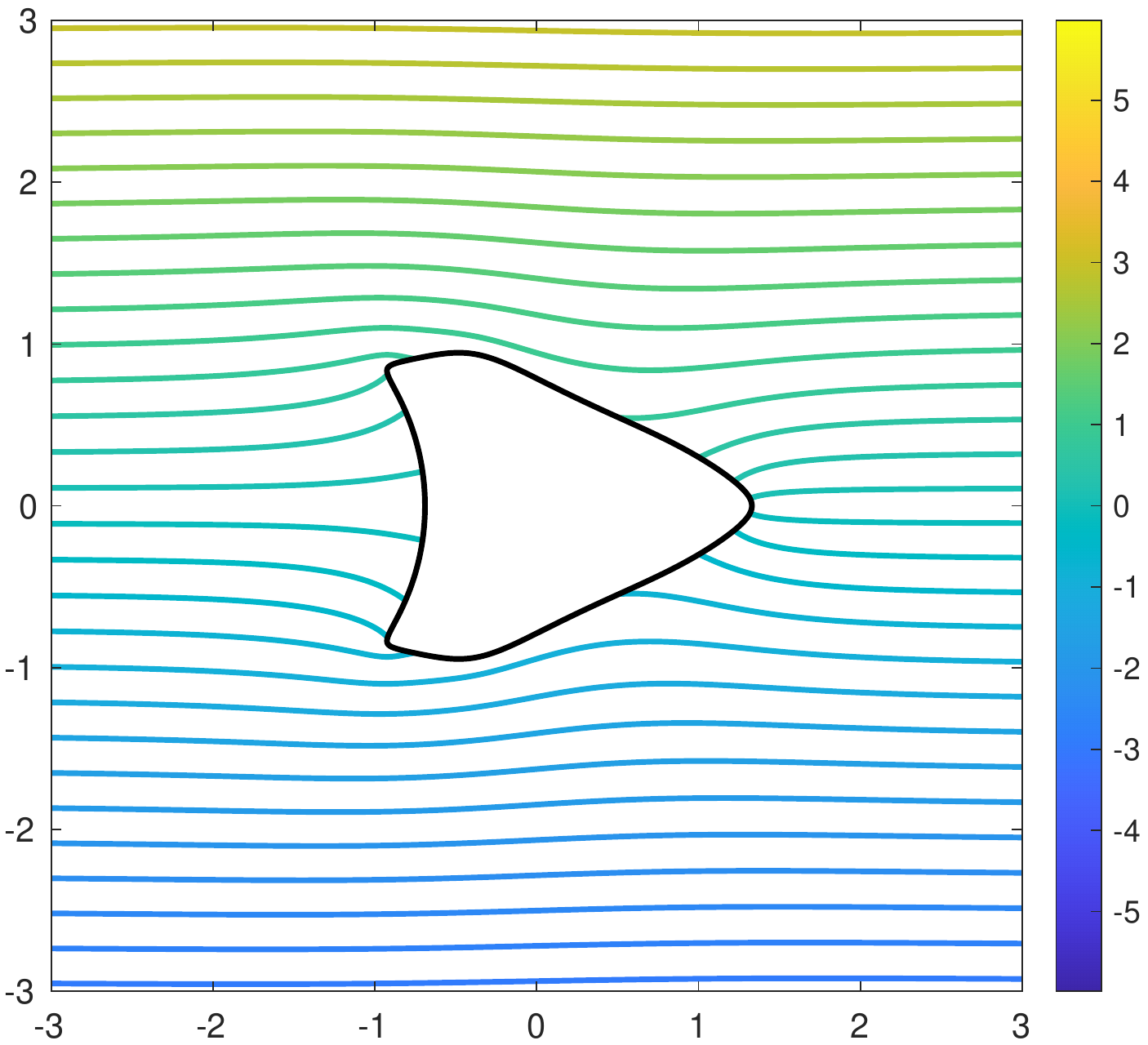}
\end{minipage}\hspace*{1mm}
\begin{minipage}{0.33\linewidth}
\includegraphics[width=\linewidth, trim={33mm 84mm 34mm 83mm}, clip]{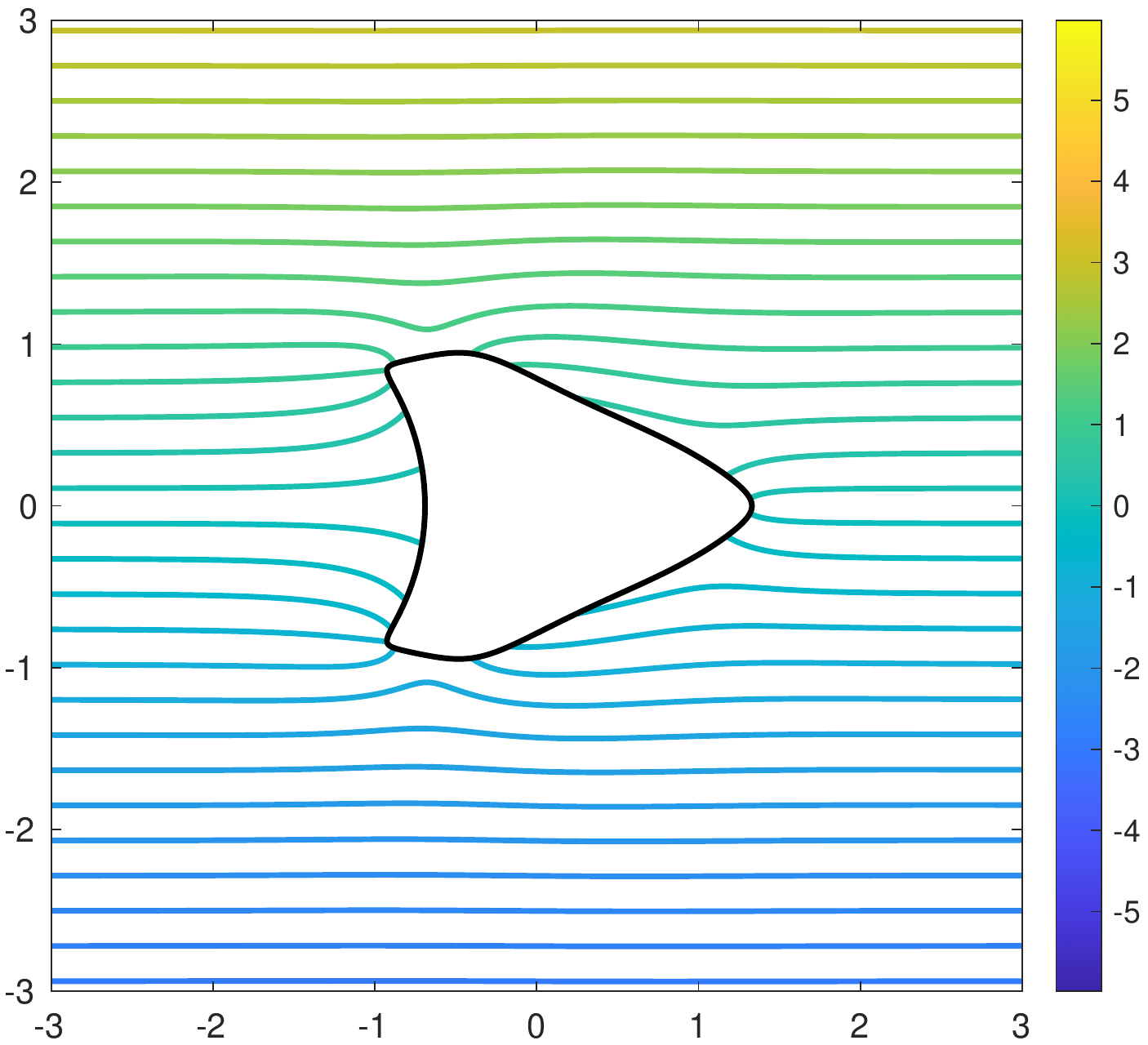}
\end{minipage}\hspace*{1mm}
\end{subfigure}
\caption{The background field is given by $H(x)=x_1$ (first row) and $H(x)=x_2$ (second row). Colored curves represent contours of the potential $u$. We set $\gamma = 1$ and $\sigma_c = 100$.}\label{weaklyneutral2}
\end{figure}

{\renewcommand{\arraystretch}{1.7}
\begin{table}[H]
\centering
\begin{tabular}{|c|c|c|c|c|}
\hline
& $\NN_{11}^{(1)}$ & $\NN_{12}^{(1)}$  & $\NN_{11}^{(2)}$ & $\NN_{12}^{(2)}$ \\
\hline\hline
Perfectly bonding & 1.2018 & 5.9817  & 12.2216 & $0.0192$ \\
\hline
1st order vanishing & $2.0709\times 10^{-16}$ & 6.2342 & $5.3517\times 10^{-15}$ & $0.0085$ \\
\hline
2nd order vanishing & $-1.6295\times 10^{-15}$& $-6.7032\times 10^{-15}$ & $9.2429\times 10^{-15}$ & $-1.1772\times 10^{-15}$ \\
\hline
\end{tabular}
\caption{The leading GPTs for the domains in Figure \ref{weaklyneutral2}.}\label{table_kite}
\end{table}
}

\section{Conclusion}

We propose a novel concept of GPT-vanishing structures for inclusions with arbitrary conductivity and imperfect interfaces. We show that inclusions of general shape can achieve nearly neutral behavior for the background field by carefully selecting a proper imperfect interface parameter. We employ exterior conformal mappings and FPT techniques to compute the coefficients of the interface parameter. Furthermore, we provide two distinct examples that highlight the effect of the vanishing FPTs, demonstrating the practical implications of our proposed concept.


\end{document}